\documentclass[11
pt,oneside,english,reqno]{amsart}
\usepackage[T1]{fontenc}
\usepackage[latin9]{inputenc}
\usepackage{geometry}
\usepackage{mathrsfs}
\usepackage{amsmath}
\usepackage{amstext}
\usepackage{hyperref}
\hypersetup{
     colorlinks   = true,
     citecolor    = blue,
     linkcolor    = blue
}

\usepackage{amsthm}
\usepackage{amssymb}
\usepackage{setspace}
\usepackage{color}
\usepackage{verbatim}

\usepackage{enumitem}

\makeatletter

\numberwithin{equation}{section}
\numberwithin{figure}{section}
\theoremstyle{plain}
\newtheorem{thm}{\protect\theoremname}
\theoremstyle{plain}

\theoremstyle{definition}
\newtheorem{defn}[thm]{\protect\definitionname}
\theoremstyle{remark}
\newtheorem{rem}[thm]{\protect\remarkname}
\theoremstyle{plain}
\newtheorem{example}{Example}
\theoremstyle{plain}
\newtheorem{lem}[thm]{\protect\lemmaname}
\theoremstyle{plain}
\newtheorem{cor}[thm]{\protect\corollaryname}
\theoremstyle{plain}
\newtheorem{prop}[thm]{\protect\propositionname}
\theoremstyle{plain}

\theoremstyle{plain}

\newcommand{\vertiii}[1]{{\left\vert\kern-0.25ex\left\vert\kern-0.25ex\left\vert #1 
    \right\vert\kern-0.25ex\right\vert\kern-0.25ex\right\vert}}

\newcommand{\ra}{\rangle}
\newcommand{\la}{\langle}

\newcommand{\cC}{\mathcal{C}}

\newcommand{\cI}{\mathcal{I}}

\newcommand{\cK}{\mathcal{K}}
\newcommand{\cL}{\mathcal{L}}

\newcommand{\cN}{\mathcal{N}}

\newcommand{\cP}{\mathcal{P}}
\newcommand{\cQ}{\mathcal{Q}}

\newcommand{\cV}{\mathcal{V}}

\newcommand{\BB}{\mathbb{B}}

\newcommand{\EE}{\mathbb{E}}

\newcommand{\RR}{\mathbb{R}}

\newcommand{\XX}{\mathbb{X}}

\usepackage{babel}
\providecommand{\conjecturename}{Conjecture}
\providecommand{\corollaryname}{Corollary}
\providecommand{\definitionname}{Definition}
\providecommand{\lemmaname}{Lemma}
\providecommand{\propositionname}{Proposition}
\providecommand{\remarkname}{Remark}
\providecommand{\theoremname}{Theorem}
\providecommand{\hypothesisname}{Hypothesis}

\providecommand{\assumptionname}{Example}

\def\@settitle{\begin{center}%
  \baselineskip14\p@\relax
  \bfseries
  \uppercasenonmath\@title
  \@title
  \ifx\@subtitle\@empty\else
     \\[1ex]\uppercasenonmath\@subtitle
     \footnotesize\mdseries\@subtitle
  \fi
  \end{center}%
}
\def\subtitle#1{\gdef\@subtitle{#1}}
\def\@subtitle{}
\makeatother

\allowdisplaybreaks
\begin{document}

\title[Infinite Dimensional pathwise Volterra processes]{Infinite Dimensional  pathwise Volterra Processes driven by Gaussian noise}
\subtitle{{\em -- Probabilistic properties and applications --}}

\author{Fred E. Benth  \and Fabian A. Harang }
\begin{abstract}
We investigate the probabilistic and analytic properties of Volterra processes constructed as pathwise integrals of deterministic kernels with respect to the H\"older continuous trajectories  of Hilbert-valued Gaussian processes.
To this end, we extend
the Volterra sewing lemma from \cite{HarTind} to the two dimensional case, in order to construct two dimensional operator-valued Volterra integrals of Young type. We prove that the covariance operator associated to infinite dimensional Volterra processes can be represented by such a two dimensional integral, which extends the current notion of representation for  such covariance operators.   We then discuss a series of applications of these results, including the construction of a rough path associated to a Volterra process driven by Gaussian noise with possibly irregular covariance structures, as well as a description of the irregular covariance structure arising from Gaussian processes time-shifted along irregular trajectories. Furthermore, we consider an infinite dimensional fractional Ornstein-Uhlenbeck process driven by Gaussian noise, which can be seen as an extension of the volatility model proposed by Rosenbaum et al. in \cite{EuchRosen}.
\end{abstract}


\thanks{ 
\\
\emph{ Acknowledgments}: FAH gratefully acknowledges financial support from the STORM project 274410, funded by the Research Council of Norway.}

\address{Fred E. Benth: Department of Mathematics, University of Oslo, P.O. box 1053, Blindern, 0316, OSLO, Norway}

\address{Fabian A. Harang: Department of Mathematics, University of Oslo, P.O. box 1053, Blindern, 0316, OSLO, Norway}

\maketitle
{
\hypersetup{linkcolor=black}
\tableofcontents
}

\section{Introduction}

Volterra processes appear naturally in models with non-local features.
In this article, we will investigate the analytic and probabilistic properties of Volterra processes constructed as pathwise integrals of a kernel $K$  against a Gaussian process $W$. For generality, we will assume that the Gaussian process $W$  takes values in a Hilbert space $H$ with a covariance operator $Q_W$, and that the kernel $(t,s)\mapsto K(t,s)$  for $t>s$ is a linear operator on the same Hilbert space. In particular, we define the process $X:[0,T]\rightarrow H$ formally by the integral
\begin{equation}\label{f v p}
 X(t)=\int_0^t K(t,s)dW(s). 
\end{equation}
At a discrete level, one can think of this process as assigning different weights through the kernel $K$ to the increments of $W$. Volterra processes have received much attention in the field of stochastic analysis over the past decades. The canonical examples are the Ornstein-Uhlenbeck process where $K(t,s)=\exp(-\alpha(t-s))$ and $W$ a Brownian motion, or the fractional Brownian motion where $K(t,s)=(t-s)^{H-\frac{1}{2}}$ and $W$ is a Brownian motion. These processes are typically used to model phenomena where some sort of memory is inherent in the dynamics, and  applications are found in various fields ranging from physics and turbulence modelling \cite{BarndorffSchmiege2007} to biology \cite{PangPardoux2020} and financial mathematics \cite{GathJaiRosen2018,Benth_2020}. See also \cite{BBV} and the references therein for an introduction to these processes and their applications.

In order to make sense of the integral appearing on the right-hand side of \eqref{f v p}, one must assume some type of regularity conditions on $K$ and $W$. The type of regularity conditions needed, typically depends on the choice of integral that is used in the construction of $X$. For example,  if $W$ is a $Q_W$-Wiener process (the infinite dimensional extension of the classical Brownian motion) one would need that $K$ is (Bochner)  square integrable in time $s\mapsto K(t,s)$  up to (and including) $t$ (see e.g. \cite{RockLiu}). 
However, for general real-valued Gaussian processes, this is not a sufficient criterion. Indeed, in the case of general Gaussian processes on the real line, it is well known that the Volterra processes appearing in \eqref{f v p} makes sense as a Wiener integral if 
\begin{equation}\label{covar rep smooth}
    \int_0 ^T \int_0 ^T K(T,r)K(T,r')\frac{\partial^2}{\partial r\partial r'}Q_W(r,r')dr dr' <\infty,
\end{equation} 
where  $Q_W$ is the real valued covariance of the Gaussian process $W$ (see e.g. \cite{HuCam}). This construction requires of course that $Q_W$ is differentiable in both variables, or at least  of bounded variation simultaneously in both variables, which excludes several interesting Gaussian processes (particular examples of which  will be discussed in detail later). An extension of the above condition to the infinite dimensional setting when $W$ is an Hilbert-valued process is quite straightforward, but one would still require strong regularity of the covariance operator $Q_W$ (say Fr\'echet differentiable). In several interesting examples, such  regularity requirements on the covariance operator are too strong. For example, consider a Gaussian  processes $(B(t))_{t\in [0,T]}$  time-shifted along an irregular (possibly deterministic) path $(Z(t))_{t\in [0,T]}$,  given as the composition process $(B(Z(t)))_{t\in [0,T]}$. The regularity of the covariance would typically be given as the composition of the regularity of the covariance associated to $B$ and the regularity of $Z$. Thus if $t\mapsto Z(t)$ is only H\"older continuous, one would not expect to get better regularity of the covariance than that of $Z$. 
The canonical example of such processes is the iterated Brownian motion given as 
\begin{equation*}
\BB(t,\omega_1,\omega_2) = B^1(\omega_1,|B^2(\omega_2,t)|)
\end{equation*}
where $B^1:[0,T]\times \Omega_1\rightarrow \RR$ and $B^2:[0,T]\times \Omega_2\rightarrow \RR$ are two independent Brownian motions on the real line.  Such processes have received much attention due to their curious probabilistic properties, as well as applications towards modelling of diffusions in cracks \cite{orsingher2009,Burdzy1993,BurdzyKhoshnevisan1998}. If we now fix a trajectory of $B^2$, it is readily seen that $\BB(\cdot,\omega_2)$ is Gaussian, with covariance given by 
\begin{equation*}
Q_{\BB(\omega_2)}(t,s)={\rm min}(B^2(t,\omega),B^2(s,\omega)), 
\end{equation*}
and therefore the regularity of $Q_{\BB}$ is inherited by the regularity of $t\mapsto B^2(t,\omega)$. 
Hence, the covariance function associated to a Volterra process driven by $\BB(\cdot,\omega_2)$ can not be constructed as in \eqref{covar rep smooth}, but an extension of this construction is needed.

In recent years,  pathwise analysis of stochastic processes has become prevalent in the literature. 
This plays a fundamental role when applying these processes in the theory of rough paths \cite{LyonsLevy,FriHai}, where analytic properties of the paths and associated "iterated integrals" constitute the main ingredients. The advantage of the rough path theory lies in the flexibility to construct pathwise solutions to controlled ODEs on the form 
\begin{equation*}
dY(t)=f(Y(t))dX(t),\qquad Y(0)=y\in H,
\end{equation*}
even when $X$ is not a semimartingale. Furthermore, one directly obtains stability in the solution mapping $(X,y)\mapsto \Gamma(X,y)$ induced by the equation above. 
The rough path theory opens for considering equations controlled by noise given as Volterra processes, which typically is of a non-semimartingale nature due to the kernel $K$. Much work has therefore been devoted to the construction of the so-called rough path above a given Volterra process driven by a Brownian motion \cite{nualart2011,Unterberger2010}. 
On the other hand, to the best of our knowledge, there is no construction of the rough path above Volterra processes driven by Gaussian noise with irregular covariance structures. In \cite[Sec. 10.2]{FriHai} the authors provide a simple criterion for the existence of a geometric rough path connected to a given Gaussian process, given that the covariance structure of this process is sufficiently regular. This requires of course the existence of a covariance function, which in the case of Volterra processes driven by Gaussian noise is given by \eqref{covar rep smooth}. A relaxation of the existence criteria for \eqref{covar rep smooth} to the case of non-smooth covariances $Q_W$ and with singular Volterra kernels $K$  is therefore needed in order to construct the rough path associated to this class of processes.

The main goal of this article is therefore to extend the sufficient conditions for construction of the covariance operator on the form of \eqref{covar rep smooth} to the case when $W$ is an infinite dimensional stochastic process and $Q_W$ is possibly nowhere differentiable in both variables. To this end, we start by giving a pathwise description of the Volterra process $X$ stated in \eqref{f v p}. Given a sample path of a Gaussian process $W$ which is $\alpha$-H\"older regular, we will show that \eqref{f v p} can be constructed in a pathwise sense through a slight modification of the newly developed Volterra Sewing Lemma from \cite{HarTind}. In this way, one directly obtains the regularity of the process $X$ as the composition of the possible singularity coming from $K$ and the regularity of $W$. 
On a heuristic level, if the kernel $K(t,s)$ is behaving locally  like $(t-s)^{-\eta}$ for $t\sim s$, and the Gaussian process has H\"older continuous trajectories of order $\alpha\in (0,1)$, then
\begin{equation}\label{a min e}
    |K(t,s)(W(t)-W(s))|_H \lesssim  (t-s)^{\alpha-\eta},
\end{equation}
and henceforth this composition is only finite for $t\rightarrow s$ whenever $\alpha-\eta>0$. 
Recalling that both in the classical probabilistic framework and in the modern approach of rough path theory, one would construct the integral as the limit when the mesh size of the partition $\mathcal{P}$ of $[0,t]$ goes to zero in the Riemann-type sum (in either $L^2(\Omega)$ if possible, or pathwise topology induced by variation or H\"older norms, see e.g. \cite{HarTind}) 
\begin{equation}\label{sum a min e}
    \sum_{[u,v]\in \mathcal{P}} K(t,u)(W(v)-W(u)). 
\end{equation}
Thus at first glance, in order for this sum to converge, it seems natural to require $\alpha-\eta>0$ to (at least) avoid any explosions when the mesh of the partition $\mathcal{P}$ goes to $0$. 

 We next extend the Volterra Sewing Lemma to the two dimensional case, in order to construct two dimensional operator-valued Volterra  integrals on the form  
\begin{equation}\label{covariance}
    \bar{Q}:=\int_0^T\int_0^{T'} K(T,r)d^2Q(r,r')K(T',r')^*
\end{equation}
for linear operators  $K$ and $Q$ on the Hilbert space $H$, and where $(s,t)\mapsto K(t,s)$ is possibly singular on the diagonal as described above. In this expression, $K^*$ is the adjoint operator of $K$, and the ordering of the integral appears naturally when considering operator-valued integrands which might be non-commutative. Our construction is based on Young-type integration theory with Volterra kernels, and only requires that $Q$ is H\"older regular, and $K$ does not blow up too fast at its singular point(s). In particular, we do not assume that $Q$ needs to be differentiable nor of bounded variation, and thus our construction truly extends the notion of the integral given in \eqref{covar rep smooth}. 
An immediate consequence of our construction is stability of the two dimensional Volterra integral with respect to changes in the Volterra operator $K$ and the operator $Q$. 

Through a consideration of the characteristic functional associated to the Volterra processes \eqref{f v p}, we next show that when $Q=Q_W$ is of sufficient regularity, then the covariance operator $Q_X$ associated the Volterra process $X$ from \eqref{f v p} is given by  $\bar{Q}$ in \eqref{covariance}. 
In the end we discuss several application areas of our results, including an analysis of the covariance structure arising from general Gaussian Volterra iterated processes, the construction of the rough path associated to Volterra processes driven by Gaussian processes with irregular covariance structures, as well as a representation of the covariance structure of certain linear fractional stochastic differential equations of Ornstein-Uhlenbeck type in Hilbert space. In the last example, we discuss the potential application towards rough volatility modelling, proposing an extension of the rough Heston model to infinite dimensions.   

Already in 2002, Towghi \cite{Towghi2002} proved that for two functions $f,g:[0,T]^2\rightarrow \RR$, the following integral  
\begin{equation}\label{eq:Young 2D}
 \int_{[0,T]^2} f(s,t)dQ_W(s,t)
\end{equation}
makes sense as the limit of a two dimensional Riemann type sum under suitable assumptions of complementary regularity between $f$ and $g$, and can thus be seen as an extension of the classical Young integral developed in \cite{Young}. The construction of the integral \eqref{covariance} can therefore be seen as an infinite dimensional extension of \eqref{eq:Young 2D} to the case when the the integrand is given as a Volterra operator of singular type.
 In the case when the covariance $Q_W$ itself is a real-valued covariance function associated to a Volterra process, Lim provided in \cite{Lim2020} a relaxation of the complementary regularity conditions originally proposed in \cite{Towghi2002} for existence of the two dimensional integral in \eqref{eq:Young 2D}. There, $f$ is assumed to be a sufficiently regular function, and thus singular Volterra kernels as we consider here fall outside of the scope of  that article. Furthermore, in the current article,  we do not impose any further structure on the covariance operator, other than regularity to keep it as general as possible, which in the end will prove useful in applications.

\subsection{Outline of the article}
The article is structured  into the following sections: 
\begin{enumerate}
\item[Sec. 2] We give an introductory account of Gaussian processes in Hilbert space, as well as continuity of the trajectories. 
\item[Sec. 3]  We introduce the concept of Volterra paths (as described in \cite{HarTind}), and provide a pathwise construction of Gaussian Volterra paths from the regularity of the trajectories of the Gaussian noise as well as the possible singularity of the Volterra kernel. 
\item[Sec. 4] This section is devoted to prove that the pathwise Volterra processes driven by Gaussian processes, are again Gaussian. To this end, we show the construction of a general covariance operator, even when the covariance function of the driving noise is nowhere differentiable.
\item[Sec. 5] We discuss a series of applications, including the construction of a rough path above Volterra processes driven by Gaussian noise with irregular covariance structures, and compute explicitly the covariance operators to various well-known Volterra processes, driven by Gaussian noise.  
\end{enumerate}
We provide background material on fractional calculus and some proofs of auxiliary results in two appendices.

\subsection{Notation}\label{subsec:notation}

We assume $(\Omega,\mathcal F, \mathbb P)$ to be a complete probability space equipped with a filtration $(\mathcal F_t)_{t\geq 0}$ satisfying the {\it usual hypotheses}.
We will work with a separable Hilbert space which will be denoted by $H$. 
The inner product in $H$ is denoted $\langle\cdot,\cdot\rangle_H$ with associated norm $\vert\cdot\vert_H$. 
The (Banach) space of bounded linear operators from $H$ to $E$, $E$ being another Hilbert space, is denoted $\mathcal L(H,E)$, with $\mathcal L(H):=\mathcal L(H,H)$.
Sometimes $E$ may also be a general Banach space, but this will be clear from the context.
 We will frequently use the $n$-simplex $\Delta_n^T$  over an interval $[0,T]$    defined by 
\begin{equation}\label{n-simp}
\Delta_n^T:=\{(s_1,\ldots,s_n)\in[0,T]^n\,|\,s_1\geq \ldots \geq s_n\}.
\end{equation} 
Also, define the diagonal in $[0,T]^n$ by $\mathrm{D}_{n}^T$, i.e. 
\begin{equation}\label{n-diag}
\mathrm{D}_{n}^T:=\{(s_1,\ldots,s_n)\in[0,T]\,|\,s_1=\ldots=s_n\}.
\end{equation}
We will denote by $\mathcal{C}^\gamma([0,T],H)$ the space of $\gamma$-H\"older continuous functions $f:[0,T]\rightarrow H$, with the norm $\|f\|_{\mathcal{C}^\gamma} =|f(0)|_H+\|f\|_{\gamma,[0,T]}$  where 
\begin{equation*}
\|f\|_{\gamma,[0,T]}=\sup_{(t,s)\in\Delta_{2}^T}\frac{|f(t)-f(s)|_{H}}{|t-s|^\gamma}. 
\end{equation*}
Whenever the interval $[0,T]$ is clear from the context, we will write $\|f\|_{\gamma}$ for the quantity $\|f\|_{\gamma,[0,T]}$. 
Aiming towards an analysis of possibly non-smooth (i.e., only H\"older continuous) covariance functions, we will also be working with increments of two-parameter functions. To this end, we will need to introduce some new notation. 
 Consider two points  $s=(s_1,s_2)$ and  $t=(t_1,t_2)$ in $[0,T]^2$  and a function $f:[0,T]^2\rightarrow H$. Let us denote
by $\square_{s,t}f$ the generalized (or rectangular) increment of $f$  over the rectangle  $[s,t]=[s_1,t_1]\times[s_2,t_2]\subset [0,T]^2 $ (notice the implied partial order of the variables in $s=(s_1,s_2)$ and $t=(t_1,t_2)$) given by 
\begin{equation}\label{eq:rec increment}
    \square_{s,t}f=f(t_1,t_2)-f(t_1,s_2)-f(s_1,t_2)+f(s_1,s_2).
\end{equation}
Note in particular that if $f$ has a mixed partial derivative $\partial ^2f(r_1,r_2)/\partial r_1 \partial r_2$ which is integrable over the rectangle  $[s,t]$, we have 
\begin{equation}\label{2d ftc}
     \square_{s,t}f=\int_{s_1}^{t_1}\int_{s_2}^{t_2}\frac{\partial ^2f(r_1,r_2)}{\partial r_1 \partial r_2}dr_2dr_1  .
\end{equation}
We remark in passing that in the literature $\square_{s,t}f$ is sometimes
referred to as the $f$-volume of the rectangle $[s,t]$.

\section{Gaussian Stochastic Processes in Hilbert Spaces}\label{sec: inf dim gaussian analysis}

One of the main objectives in this article is to study the regularity properties of various stochastic processes in Hilbert spaces, together with their covariance operators. In this Section we provide some background material on the important class of Gaussian stochastic processes in Hilbert space which will be at the core of our studies. %

For a Gaussian process in Hilbert space, one associates a covariance operator on the Hilbert space where the process lives, which describes the covariance structure of the process.  A special case of such Gaussian processes is the $Q$-Wiener process, where the covariance operator $Q$ is a non-negative definite trace class linear operator. This process can be seen as an infinite dimensional extension of the well known Brownian motion,
as these processes share many of the same probabilistic and analytic properties. The infinite dimensional 
Wiener process is a special case of a more general class of Hilbert-valued Gaussian stochastic processes.
Below we give a general definition of Hilbert-valued Gaussian random variables, and then extend this definition to Hilbert-valued Gaussian processes.  We highlight this definition with the example of the construction of the Hilbert-valued fractional Brownian motion. 

We say that an $H$-valued random variable $X$ is square-integrable if $\mathbb E[\vert X\vert_H^2]<\infty$. 
If $X$ is square-integrable with zero mean, that is, $\mathbb E[X]=0$ where $0\in H$ is the zero element and the expectation is in the sense of Bochner integration with respect to the probability $\mathbb P$, we introduce the {\it covariance functional} $Q$ associated to $X$ by
\begin{equation*}
    Q=\mathbb E[X\otimes X].
\end{equation*}
Here, $\otimes$ is the tensor product such that for any $g,h,x\in H$, $(g\otimes h)(x)=\langle g,x\rangle_H h$. Note that by square-integrability of $X$, the expectation defining $Q$ is well-defined as a Bochner integral. It is known that $Q\in\mathcal L(H)$ is a symmetric, positive semi-definite trace class operator. 
In fact, we have, $\text{Tr}(Q)=\mathbb E[\vert X\vert_H^2]$
and
$$
\mathbb E[\langle X,g\rangle_H\langle X,h\rangle_H]=\langle Qg,h\rangle_H,
$$
for any $g,h\in H$. 
We have the following standard definition of a Gaussian 
random variable in Hilbert space: 
\begin{defn}
\label{def:Hilbert Gaussian variable}
An $H$-valued random variable $X$ is said to be {\it Gaussian}
if $\langle X,h\rangle_H$ is a real-valued Gaussian random variable for every $h\in H$.
\end{defn}
We remark that Gaussian variables in Hilbert space are square-integrable (see 
\cite[Thm. 3.31]{PesZab}). We introduce a Gaussian process in Hilbert space
by the following definition (see \cite[Def. 3.30]{PesZab}):
\begin{defn}
\label{def:Hilbert Gaussian process}
An $H$-valued stochastic process $(X(t))_{t\geq 0}$ is said to be {\it Gaussian}
if for every $n\in\mathbb N$, $0\leq t_1<t_2\cdots<t_n<\infty$, $(X(t_1),X(t_2),\ldots,X(t_n))$
is an $H^{ n}$-valued Gaussian random variable.
\end{defn}
By definition, we have that a Gaussian process
can be equivalently characterised by saying that for 
every $n\in\mathbb N$, $0\leq t_1<t_2\cdots<t_n<\infty$ and
$h_1,\ldots,h_n\in H$, 
$(\langle X(t_1),h_1\rangle_H,\ldots,\langle X(t_n),h_n\rangle_H)$ is an $n$-variate Gaussian random variable on $\mathbb R^n$. We have a covariance operator 
defined as (for $s,t\geq 0$)
\begin{equation*}
    Q(s,t):=\mathbb E[X(s)\otimes X(t)]\in\mathcal L(H).
\end{equation*}
Here we have implicitly assumed that the process has zero mean. Note that generally $Q(s,t)\neq Q(t,s)$. But,
$$
\langle Q(s,t)g,h\rangle_H=\mathbb E[\langle X(s),g\rangle_H\langle X(t),h\rangle_H]=\langle g,Q(t,s)f\rangle_H,
$$
and thus, $Q(s,t)^*=Q(t,s)$. 
But, on the other hand, $Q(t,t)$ is a positive semi-definite and symmetric trace class operator.

An important Gaussian process in Hilbert space is the $Q$-Wiener process, which has a covariance operator $Q(s,t)=Q\min(s,t)$ where $Q$ is a symmetric positive definite trace class operator. 
As in \cite{TindelTudorViens2003,DuncanPasikMaslowski2002,GRECKSCH1999}
we can define a $Q$-fractional Brownian motion with values in 
Hilbert space by letting
\begin{equation}\label{eq:covar-fbm}
    Q(s,t):=R^h(s,t)Q,
\end{equation}
for a symmetric positive definite trace class operator $Q$ and the real-valued function
\begin{equation}
\label{eq:fbm-r-func}
    R^h(s,t)=\frac12\left(s^{2h}+t^{2h}-\vert t-s\vert^{2h}\right),
\end{equation}
with the Hurst index $h\in(0,1)$ and $s,t\geq 0$. Letting $h=0.5$, the $Q$-fractional Brownian motion is a $Q$-Wiener process. 

In our analysis, the continuity properties of paths play an important role. For this purpose, we recall the Kolmogorov continuity theorem (see e.g. \cite[Thm. 3.3]{DaPraZab}, where a full proof of the below statement can be found).
\begin{thm}
\label{thm:Kolmogorov}
{\rm (Kolmogorov's continuity theorem)} Let $W:\Omega \times [0,T]\rightarrow H$ be a stochastic process 
such that for some positive constants $C>0 $, $\epsilon>0$, $\delta>1$ and all $(t,s)\in \Delta_2^T$ the following inequality holds
\begin{equation*}
    \mathbb{E}\left[\vert W(t)-W(s)\vert_H^\delta\right] \leq C|t-s|^{1+\epsilon}.
\end{equation*}
Then there exists a pathwise continuous modification $\widetilde{W}$ of $W$. More
specifically, 
the mapping  $t\mapsto \widetilde{W}(\omega,t)$ is $\alpha$-H\"older continuous with $\alpha=\frac{\epsilon}{\delta}$, $\mathbb{P}-a.s$. 
\end{thm}

For a $Q$-Wiener process, we readily see that 
\begin{equation*}
    \mathbb E[\vert W(t)-W(s)\vert_H^2]=\vert t-s\vert \text{Tr}(Q),
\end{equation*}
while for the fractional Brownian motion with covariance operator
defined in \eqref{eq:covar-fbm} we have
\begin{equation*}
    \mathbb E[\vert W(t)-W(s)\vert_H^2]=\vert t-s\vert^{2h} \text{Tr}(Q).
\end{equation*}
We have the following result on the H\"older continuity of
the fractional Brownian motion (which seems to be known but we include a proof for the convenience of the reader):
\begin{prop}\label{Prop: Holder cont of fbm}
Let $W$ be a $Q$-fractional Brownian motion 
with values in $H$ and 
covariance operator given in 
\eqref{eq:covar-fbm} with Hurst parameter $h\in(0,1)$. Then,
for $(t,s)\in\Delta_2^T$ 
$$
\mathbb E[\vert W(t)-W(s)\vert_H^{2n}]\leq \vert t-s\vert^{2hn}({\rm Tr}(Q))^n\mathbb E[Z^{2n}]
$$
for any $n\in\mathbb N$ and with $Z$ being a standard normal random variable in $\RR$.
Moreover, there exists a version of $W$ which 
is H\"older continuous of order $\alpha<h$,
$\mathbb P-$ a.s.

\end{prop}
\begin{proof}
Let $(e_i)_{i\in\mathbb N}$ be the ONB of 
eigenvectors of $Q$, with the covariance operator $Q(s,t)$ of $W$ defined in \eqref{eq:covar-fbm}. We have that $W(t)-W(s)$ is a Gaussian
mean-zero random variable, and a straightforward calculation yields 
that it has the covariance operator $\vert t-s\vert^{2h} Q$. Thus,
$X_i:=\langle W(t)-W(s),e_i\rangle_H$ is a mean-zero real-valued Gaussian random variable, with variance equal to $\vert t-s\vert^{2h}\lambda_i$. Here, $\lambda_i>0$ is the 
$i$th eigenvalue of $Q$. As $(e_i)_{i\in\mathbb N}$ are the eigenvectors of $Q$, $X_i$ is independent of $X_j$ for any $i\neq j, i.j\in\mathbb N$. Let $(Z_i)_{i\in\mathbb N}$ be a sequence of independent identically distributed real valued standard normal variables. Then, in distribution, we have
$X_i=\vert t-s\vert^h\sqrt{\lambda_i}Z_i$.
Parseval's equality yields
\begin{align*}
\mathbb E[\vert W(t)-W(s)\vert_H^{2n}]&=\mathbb E\left[\left(\sum_{i=1}^{\infty}\langle W(t)-W(s),e_i\rangle_H^2\right) ^n\right] \\
&=\vert t-s\vert^{2hn}\mathbb E\left[\left(\sum_{i=1}^{\infty}\lambda_i Z_i^2\right)^n\right].
\end{align*}
If $n=1$, we are done. Suppose that $n\geq 2$. For $p>1$ and $q$ being the reciprocal of $p$, we find by H\"older's inequality
\begin{align*}
    \sum_{i=1}^{\infty}\lambda_i Z_i^2&=\sum_{i=1}^{\infty}\lambda_i^{1/q}\lambda_i^{1/p}Z_i^2 \\
    &\leq\left(\sum_{i=1}^{\infty}\lambda_i\right)^{1/q}
    \left(\sum_{i=1}^{\infty}\lambda_iZ_i^{2p}\right)^{1/p} \\
    &=(\text{Tr}(Q))^{1/q}\left(\sum_{i=1}^{\infty}\lambda_iZ_i^{2p}\right)^{1/p}.
\end{align*}
Choosing $p=n>1$ and $q=n/n-1$, we find
\begin{align*}
   \mathbb E[\vert W(t)-W(s)\vert_H^{2n}]&\leq\vert t-s\vert^{2hn}(\text{Tr}(Q))^{n-1}\sum_{i=1}^{\infty}\lambda_i\mathbb E[Z_i^{2n}], 
\end{align*}
and the first result of the Proposition follows. 

For the second conclusion, 
suppose that $n\in\mathbb N$ is such that $2hn>1$. Then we obtain existence of an $\alpha:=h-\frac1{2n}$ H\" older continuous version of $W$ from
Kolmogorov's continuity theorem \ref{thm:Kolmogorov}. As $n$ can be chosen arbitrary large, we conclude that there exists a version of $W$ which is  H\"older continuous of order $\alpha<h$, $\mathbb P$-a.s.

\end{proof}
As a simple consequence of the above, we see that a $Q$-Wiener process
has a version with H\"older continuous paths of order $\alpha<1/2$.  
In the analysis that follows in the next sections, we will make use of 
processes with specific regularity properties of the paths. The discussion in this Section shows that we have available specific cases of (Gaussian) stochastic processes with various H\"older regularity
of the paths. Gaussian processes will constitute our canonical class of models, and whenever we refer to such processes we will have their 
H\"older continuous version in mind.

\section{Pathwise Volterra processes in Hilbert Spaces}

In this Section we introduce and study Volterra processes of the form \eqref{f v p}. In order to give a pathwise description of Volterra integrals driven by generic H\"older paths, 
we will apply a variant of the celebrated Sewing Lemma from the theory of rough paths, modified to accommodate the Volterra structure inherit in the processes of interest. 
This lemma was first proved in \cite{HarTind}  where the authors extend aspects of the theory of rough paths to the analysis of Volterra equations with singular kernels driven by irregular paths.  
In order to discuss Volterra integration in a pathwise manner, we will introduce an abstract space of Volterra paths, 
as defined in \cite{HarTind}. This definition allows us to discuss the continuity properties of Volterra paths,
independent of the Volterra integral representation. However, it will be instructive for the reader to think of the expression 
\begin{equation}\label{vp}
X^{\tau}(t):=X(\tau,t)=\int_{0}^{t}K(\tau,r)dW(r), 
\end{equation}
where we have chosen to let a Volterra process have two arguments by decoupling the first argument $\tau$ in the kernel, and the upper integration parameter $t$, with $\tau\geq t$. 
The classical Volterra process  is of course given by the mapping  $t\mapsto X^t(t)$ (recall \eqref{f v p}). Thus,  if $W:[0,T]\rightarrow H$  is a smooth path,\footnote{Notice here that $W$ is a general path, not necessarily a Gaussian process as we discussed in the previous section. However, in typical cases we have $W$ being a $Q$-Wiener or fractional Brownian motion, which explains why we use the notation $W$.} then the integral in \eqref{vp} can be interpreted in the Riemann sense, provided that the kernel $K\in\mathcal{L}(H)$ is Riemann integrable with respect to $W$, and thus we can view $X$ as a path from   $\Delta_2^T$ to $H$. 
Note that  we can then measure the regularity of  $X$   in both $t$  and $\tau$ separately,  where at least at a heuristic level, the regularity of $X$   in the $\tau$ parameter can be expected to be inherited from the regularity of the kernel $K$  in $\tau$. On the other hand the regularity of $X$  in the $t$ parameter will typically be inherited by the path $W$.


\begin{defn}\label{Volterra Holder Space}
Let $\gamma,\eta \in(0,1)$ and assume $\gamma-\eta>0$. We denote by $\mathcal{V}^{(\gamma,\eta)}(\Delta_{2}^{T},H)$  the space of all functions $f:\Delta_{2}^T \rightarrow H$  such that 
\begin{equation*}
    \|f\|_{(\gamma,\eta)}:=\|f\|_{(\gamma,\eta),1}+\|f\|_{(\gamma,\eta),1,2}<\infty
\end{equation*}
where we define the semi-norms by
\begin{equation}\label{eq:Volterra holder norms}
\begin{aligned}
\|f\|_{(\gamma,\eta),1}&:=\sup_{\left(\tau,t,s\right)\in\Delta_{3}^T}\frac{\vert f^{\tau}(t)-f^\tau(s)\vert_H}{[|\tau-t|^{-\eta}|t-s|^{\gamma}]\wedge |\tau-s|^{\gamma-\eta}}
\\
\|f\|_{(\gamma,\eta),1,2}&:=\sup_{\substack{\left(\tau',\tau,t,s\right)\in\Delta_{4}^T \\ \theta\in [0,1],\zeta\in [0,\gamma-\eta)}}\frac{\vert f^{\tau'}(t)-f^\tau(t)-f^{\tau'}(s)+f^\tau(s)\vert_H}{|\tau'-\tau|^{\theta}|\tau-t|^{-\theta+\zeta}\left\{[|\tau-t|^{-\eta-\zeta}|t-s|^{\gamma}]\wedge |\tau-s|^{\gamma-\eta-\zeta}\right\}}.
\end{aligned}
\end{equation}
Here we have used the notation $f^{\tau}(t):=f(\tau,t)$ for $(\tau,t)\in\Delta_2^T$. 
\end{defn}

\begin{rem}
Consider a subspace $\hat{\cV}^{(\gamma,\eta)}(\Delta_2^T,H) \subset \cV^{(\gamma,\eta)}(\Delta_2^T,H)$ containing all Volterra paths $f\in \cV^{(\gamma,\eta)}(\Delta_2^T,H)$ such that $f_0:=f^\tau(0)=c\in H$ for all $\tau\in [0,T]$.   Under the norm 
\begin{equation*}
   \|f\|_{(\gamma,\eta),*}:=|f_0|_H+\|f\|_{(\gamma,\eta)}
\end{equation*}  
the space $\hat{\cV}^{(\gamma,\eta)}(\Delta_2^T,H)$ is a Banach space, see e.g. \cite{HarTind}.  
\end{rem}

\begin{rem}\label{rem: two variable extension of Votlerra-Holder space}
We can extend the definition of $\mathcal{V}^{(\gamma,\eta)}(\Delta_2^T,H)$ above to functions $f:\Delta_3^T\rightarrow H$, where $f$ has one upper variable and two lower variables, that is, 
\begin{equation*}
    (\tau,t,s)\mapsto f^\tau(t,s).
\end{equation*}
In this case, we consider the semi-norms $\|f\|_{(\gamma,\eta),1}$  and $\|f\|_{(\gamma,\eta),1,2}$ to be given by 
\begin{align*}
\|f\|_{(\gamma,\eta),1}&:=\sup_{\left(\tau,t,s\right)\in\Delta_{3}^T}\frac{\vert f^{\tau}(t,s)\vert_H}{|\tau-t|^{-\eta}|t-s|^{\gamma}\wedge |\tau-s|^{\gamma-\eta}}
\\
\|f\|_{(\gamma,\eta),1,2}&:=\sup_{\substack{\left(\tau',\tau,t,s\right)\in\Delta_{4}^T \\ \theta \in [0,1],\zeta\in [0,\gamma-\eta) }}\frac{\vert f^{\tau'}(t,s)-f^\tau(t,s)\vert_H}{|\tau'-\tau|^{\theta}|\tau-t|^{-\theta+\zeta}\left[|\tau-t|^{-\eta-\zeta}|t-s|^{\gamma}\wedge |\tau-s|^{\gamma-\eta-\zeta}\right]}.
\end{align*}
We denote the space of such three-variable functions by $\mathcal{V}^{(\gamma,\eta)}_3(\Delta_3^T,H)$.
\end{rem}

The next proposition shows the relation between the space of classical H\"older paths $\mathcal{C}^\rho([0,T],H)$  and $\mathcal{V}^{(\gamma,\eta)}(\Delta_2^T,H)$  when $\gamma-\eta=\rho>0$. 

\begin{prop}\label{prop: Volterra implies Holder}
Suppose 
$f\in \mathcal{V}^{(\gamma,\eta)}(\Delta_2^T,H)$ 
with $\gamma-\eta=\rho>0$ and $f^\tau(0)=c\in H$ is constant (in $H$) for all $\tau\in [0,T]$. Then the restriction of $\tilde{f}(t):=f^t(t)$  of $f$ to the diagonal of $\Delta_2^T$ is $\zeta$-H\"older continuous for any $\zeta\in [0,\rho)$, i.e.  $\tilde{f}\in \mathcal{C}^{\zeta}([0,T],H)$. 
\end{prop}
\begin{proof}
By assumption it follows that $f^t(0)-f^s(0)=0$ for all $s,t\in [0,T]$. Furthermore, by definition of the norms, we have that
\begin{equation*}
\begin{aligned}
    |f^t(t)-f^s(s)|_H& \leq |f^t(t)-f^t(s)|_H+|f^t(s)-f^s(s)|_H\\
   &\leq |f^t(t)-f^t(s)|_H + |f^t(s)-f^s(s) -f^t(0)+f^s(0)|_H
   \\
    &\leq \|f\|_{(\gamma,\eta),1}|t-s|^{\rho}+ T^{\gamma-\eta-\zeta}\|f\|_{(\gamma,\eta),1,2}|t-s|^{\zeta}.
    \end{aligned}
\end{equation*}
In the second majorization of the last inequality above, we applied the definition
of $\|\cdot\|_{(\gamma,\eta),1,2}$ in \eqref{eq:Volterra holder norms} with $\theta=\zeta\in [0,\gamma-\eta)$, i.e., 
for any $(\tau',\tau,t,s)\in\Delta_4^T$, the following relation holds
\begin{equation*}
    |f^{\tau'}(t)-f^\tau(t) -f^{\tau'}(s)+f^\tau(s)|_H\leq \|f\|_{(\gamma,\eta),1,2} |\tau'-\tau|^\zeta|\tau-t|^0|\tau-s|^{\gamma-\eta-\zeta}.
\end{equation*}
Thus, for $(\tau',\tau,t,s):=(t,s,s,0)$ we get the desired inequality after observing that 
$s^{\gamma-\eta-\zeta}\leq T^{\gamma-\eta-\zeta}$. 
As $\zeta\in [0,\rho)$ is arbitrary, we see that $\tilde{f}\in \cC^{\zeta}([0,T],H)$ and the result follows.  
\end{proof}

In order to accommodate pathwise Volterra integrals, we will need a modified version of the Sewing Lemma.  But first we define a suitable space of abstract Volterra integrands. In the sequel, we will work with integrals taking values in a space of linear operators on Hilbert spaces. We therefore state the Volterra Sewing Lemma in general Banach spaces.

\begin{defn}\label{abstract integrnds}
 Consider a Banach space $E$, and suppose $\gamma,\eta\in (0,1)$, $\beta\in (1,\infty)$ and $\kappa\in (0,1)$ is such that the following relations hold~$\beta-\kappa\geq \gamma-\eta>0$.
Denote by  $\mathscr{V}^{(\gamma,\eta)(\beta,\kappa)}\left(\Delta_{3}^T,E\right)$,
the space of all functions $\Xi:\Delta_{3}^T \rightarrow E$
such that 
\begin{equation}\label{abstract integrand space}
\vertiii{\Xi}_{(\gamma,\eta)(\beta,\kappa)}:=\|\Xi\|_{\left(\gamma,\eta\right)}+\vertiii{\delta \Xi}_{\left(\beta,\kappa\right)}<\infty.
\end{equation}
Here $\delta$  is the operator defined for any $s\leq u\leq t\leq \tau$   acting on functions $g$  by 
\begin{equation}\label{delta}
\delta_{u}g^\tau(t,s)=g^\tau(t,s)-g^\tau(t,u)-g^\tau(u,s).
\end{equation}
The norm $\|\Xi \|_{(\gamma,\eta)}$  is given as in Remark \ref{rem: two variable extension of Votlerra-Holder space}, while the quantity $\vertiii{\delta \Xi}_{(\beta,\kappa)}$ is a slight modification of the norms from Remark \ref{rem: two variable extension of Votlerra-Holder space} defined by
\begin{equation*}
    \vertiii{\delta\Xi}_{\left(\beta,\kappa\right)}:=\vertiii{\delta\Xi}_{\left(\beta,\kappa\right),1}+\vertiii{\delta\Xi}_{\left(\beta,\kappa\right),1,2}
\end{equation*}
where 
\begin{equation}\label{dd3}
\begin{aligned}
\vertiii{\delta\Xi}_{\left(\beta,\kappa\right),1}&:=\sup_{\left(\tau,t,u,s\right)\in\Delta_{4}^T}\frac{|\delta_{u}\Xi^{\tau}(t,s)|_E}{|\tau-t|^{-\kappa}|t-s|^{\beta}\wedge |t-s|^{\beta-\kappa}},
\\
\vertiii{\delta\Xi}_{\left(\beta,\kappa\right),1,2}&:=\sup_{\substack{\left(\tau',\tau,t,u,s\right)\in\Delta_{5}^T \\ \theta \in [0,1],\zeta\in [0,\beta-\kappa) }}\frac{|\delta_{u}\left[\Xi^{\tau'}(t,s)-\Xi^{\tau}(t,s)\right]|_E}{|\tau'-\tau|^{\theta}|\tau-t|^{-\theta+\zeta}\left[|\tau-u|^{-\kappa-\zeta}|t-s|^{\beta}\right]}.
\end{aligned}
\end{equation}
where we mean $\Xi^{\tau}(t,s):=\Xi(\tau,t,s)$. 
In the sequel we call $\mathscr{V}^{(\gamma,\eta)(\beta,\kappa)}(\Delta_3^T,E)$ the space of all abstract Volterra integrands. 
\end{defn}

We are now ready to state the Sewing Lemma adapted to Volterra integrands. The following lemma is a trivial extension of \cite[Lemma 21]{HarTind} to the case of Banach-valued Volterra kernels.

\begin{lem}\label{lem:(Volterra-sewing-lemma)}
\emph{(Volterra Sewing Lemma)}
 Let $E$ be a Banach space, and consider parameters $\gamma,\eta\in (0,1),$ $\beta\in (1,\infty)$, and $\kappa\in (0,1)$ such that  $\beta-\kappa\geq \gamma-\eta>0$. There exists a unique continuous map
$\mathcal{I}:\mathscr{V}^{(\gamma,\eta)(\beta,\kappa)}\left(\Delta_{3}^T,E\right)\rightarrow\mathcal{V}^{\left(\gamma,\eta\right)}\left(\Delta_{2}^T,E\right)$
such that the following statements holds true

\begin{itemize}[leftmargin=0.7cm]
    \setlength\itemsep{.1in}
\item[{\rm (i)}] The quantity 
$\mathcal{I}(\Xi^{\tau})(t,s):=\lim_{|\mathcal{P}|\rightarrow 0} \sum_{[u,v]\in\mathcal{P}} \Xi^{\tau}(v,u) $
exists (in $E$) for all tuples $(\tau,t,s)\in \Delta_{3}^T$, where $\mathcal{P}$  is a generic partition of $[s,t]$  and $|\mathcal{P}|$  denotes the mesh size of the partition. 

\item[{\rm (ii)}] For all $(\tau,t,s)\in \Delta_{3}^T$ the following inequality holds
\begin{equation}\label{sy lemma bound1}
\vert\mathcal{I}\left(\Xi^{\tau}\right)(t,s)-\Xi^{\tau}(t,s)\vert_E\lesssim\vertiii{\delta\Xi}_{\left(\beta,\kappa\right),1}\left(|\tau-t|^{-\kappa}|t-s|^{\beta}\wedge |\tau-s|^{\beta-\kappa}\right),
\end{equation} 

\item[{\rm (iii)}]
For all $(\tau',\tau,t,s)\in \Delta_{4}^T$,  $\theta\in[0,1]$ and $\zeta\in [0,\beta-\kappa)$, we denote by  $\Xi^{\tau',\tau}(t,s)=\Xi^{\tau'}(t,s)-\Xi^{\tau}(t,s)$, and the following inequality holds
\begin{multline}\label{sy lemma bound2}
\vert\mathcal{I}(\Xi^{\tau',\tau})(t,s)-\Xi^{\tau',\tau}(t,s)\vert_E
\\
\lesssim\vertiii{\delta\Xi}_{\left(\beta,\kappa\right),1,2}\left(|\tau'-\tau|^\theta|\tau-t|^{-\theta+\zeta}\left[|\tau-t|^{-\kappa-\zeta}|t-s|^{\beta}\wedge |\tau-s|^{\beta-\kappa-\zeta}\right]\right).
\end{multline} 
\end{itemize}
Moreover, $t\mapsto \cI(\Xi^\tau)(t):=\cI(\Xi^\tau)(t,0)$ is additive, in the sense that $\cI(\Xi^\tau)(t,s)=\cI(\Xi^\tau)(t,0)-\cI(\Xi^\tau)(s,0)$, and we conclude that $(\tau,t)\mapsto \mathcal{I}\left(\Xi^{\tau}\right)(t)\in \cV^{(\gamma,\eta)}(\Delta_2^T,E)$. 

\end{lem}

We are frequently going to work with Volterra kernels in various contexts, and we therefore state a common hypothesis on the regularity on the kernels we consider in this article. 

\begin{defn}\label{hyp}
For $\eta\in(0,1)$, suppose $K:\Delta_2^T \rightarrow \mathcal{L}(H)$ is a linear operator on the Hilbert space $H$ which satisfies for $(\tau,t,s,r)\in\Delta_4^T$ and any $\theta,\nu\in [0,1]$ the following inequalities
\begin{align}\label{bound1}
\vert K(t,s)f\vert_H &\lesssim  |t-s|^{-\eta}\vert f\vert_H
\\\label{bound2}
\vert (K(t,s)-K(t,r))f\vert_H&\lesssim  |t-s|^{-\eta-\theta}|s-r|^\theta\vert f\vert_H. 
\\ \label{bound 22}
\vert (K(\tau,s)-K(t,s))f\vert_H&\lesssim  |t-s|^{-\eta-\theta}|\tau-t|^\theta\vert f\vert_H.
\\ \label{bound 3}
\vert(K(\tau,s)-K(\tau,r)-K(t,s)+K(t,r))f\vert_H&\lesssim |\tau-r|^{-\nu-\theta-\eta}  |\tau-t|^\theta|r-s|^\nu \vert f\vert_H. 
\end{align}
for every $f\in H$.
Then we say that the kernel $K$ is a Volterra kernel of order $\eta$. We denote the space of all Volterra kernels $K$ of order $\eta\in(0,1)$ satisfying \eqref{bound1}-\eqref{bound 3} by $\mathcal{K}_{\eta}$. We equip this space with the following semi-norm
\begin{equation}\label{Knorm}
    \|K\|_{\mathcal{K}_{\eta}}:=\|K\|_{\eta,1}+\|K\|_{\eta,2}+\|K\|_{\eta,3},\|K\|_{\eta,4},
\end{equation}
where we define the three semi-norms on the right-hand side above by 
\begin{align}\label{Knorm 1}
    &\|K\|_{\eta,1}:=\sup_{(t,s) \in \Delta_2^T} \frac{\|K(t,s)\|_{\text{op}}}{|t-s|^{-\eta}},
    \\\label{Knorm 2}
   & \|K\|_{\eta,2}:=\sup_{\substack{(t,u,s) \in \Delta_3^T \\ \theta \in [0,1]}} \frac{\|K(t,s)-K(u,s)\|_{\text{op}}}{|t-u|^\theta|u-s|^{-\theta-\eta}}, 
   \\ \label{Knorm 3}
   & \|K\|_{\eta,3}:=\sup_{\substack{(t,u,s) \in \Delta_3^T \\ \theta \in [0,1]}} \frac{\|K(t,u)-K(t,s)\|_{\text{op}}}{|u-s|^\theta|t-u|^{-\theta-\eta}}, 
    \\\label{Knorm 4}
   & \|K\|_{\eta,4}:=\sup_{\substack{(\tau',\tau,s,r) \in \Delta_4^T \\ \theta,\nu \in [0,1]}} \frac{\|K(\tau',s)-K(\tau',r)-K(\tau,s)+K(\tau,r)\|_{\text{op}}}{|\tau-r|^{-\nu-\theta-\eta}  |\tau'-\tau|^\nu|r-s|^\theta},
\end{align}
with $\Vert\cdot\Vert_{\text{op}}$ denoting the operator norm.
\end{defn}

\begin{rem}
Note that if $K\in \mathcal{K}_{\eta}$, then also $K^*\in \mathcal{K}_{\eta}$. Indeed, this follows from the well-known fact that $\|K(t,s)\|_{\text{op}}=\|K^*(t,s)\|_{\text{op}}$ for any $(t,s)\in\Delta_2^T$. 
\end{rem}
\begin{rem}
We restrict our analysis here to $K(t,s)\in \cL(H)$. One could easily extend our results to operators $K(t,s)\in \cL(H,H')$ for some general Hilbert spaces $H$ and $H'$, or even to $K(t,s)\in\mathcal L(E,E')$ for some general  Banach spaces $E$ and $E'$, by adjusting the spaces of paths and functions accordingly. However, to increase readability 
we confine our considerations to $\cL(H)$.
\end{rem}
With  Definition \ref{hyp} at hand we will now show that we can construct Volterra processes from H\"older paths in a deterministic manner using the Volterra Sewing Lemma \ref{lem:(Volterra-sewing-lemma)}.

\begin{prop}\label{regularity W}
Suppose $\gamma,\eta\in(0,1)$ are such that $\gamma-\eta>0$. Let $W\in\mathcal{C}^{\gamma}\left([0,T],H\right)$ and consider a kernel $K\in \cK_{\eta}$ as introduced in Definition \ref{hyp}. Let the abstract integrand $\Xi$ be given as  
\begin{equation}\label{spec varxi}
    \Xi^{\tau}(t,s):=K(\tau,s)\left(W(t)-W(s)\right).
\end{equation}
 Then we define the pathwise Volterra process as the integral
\begin{equation}\label{fractional process}
X^\tau(t):=\int_{0}^{t}K(\tau,s)dW(s):=\mathcal{I}\left(\Xi^{\tau}\right)(t),
\end{equation}
where, for a partition $\mathcal{P}$  of $[0,t]$, the integral $\mathcal{I}\left(\Xi^{\tau}\right)$ is defined as in Lemma \ref{lem:(Volterra-sewing-lemma)} by 
\begin{equation}\label{sum int spec}
\mathcal{I}(\Xi^{\tau})(t):=\lim_{|\mathcal{P}|\rightarrow 0} \sum_{[u,v]\in\mathcal{P}} \Xi^{\tau}(v,u),
\end{equation}
and the limit is taken in $H$.
Moreover, we have that 
$(\tau,t)\mapsto X^\tau(t)\in \mathcal{V}^{(\gamma,\eta)}(\Delta_2^T,H)$. 
\end{prop}
\begin{proof}
With Lemma \ref{lem:(Volterra-sewing-lemma)} in mind, we recall that in order to show convergence of the integral in \eqref{sum int spec}, it is sufficient to prove that $\Xi$ given as in \eqref{spec varxi} satisfies the following conditions 
\begin{align*}
       \|\Xi\|_{(\gamma,\eta)}&=\|\Xi\|_{(\gamma,\eta),1}+\|\Xi\|_{(\gamma,\eta),1,2}<\infty \qquad {\rm and}
       \\
   \vertiii{\delta\Xi}_{(\beta,\kappa)} & =\vertiii{\delta\Xi}_{(\beta,\kappa),1}+\vertiii{\delta\Xi}_{(\beta,\kappa),1,2}<\infty , 
\end{align*}
for some $(\beta,\kappa)\in(1,\infty)\times[0,1)$ with $\beta-\kappa\geq\gamma-\eta$. The fact that $\|\Xi\|_{(\gamma,\eta)}<\infty$ follows directly from the assumptions on the noise $W$ and kernel $K$: Indeed, since $K\in \cK_{\eta}$ and $W\in \mathcal{C}^\gamma([0,T],H)$ yields
\begin{equation*}
    \vert  K(\tau,s)(W(t)-W(s))\vert_H\lesssim \|K\|_{\eta,1}\|W\|_{\gamma} |\tau-s|^{-\eta}|t-s|^{\gamma}.
\end{equation*}
Notice that since $\tau\geq t\geq s$ we have 
\begin{equation}\label{simple inequality}
|\tau-s|^{-\eta}|t-s|^{\gamma}\leq [|\tau-t|^{-\eta}|t-s]^{\gamma}]\wedge|\tau-s|^{\gamma-\eta}.
\end{equation}
This shows that $\Vert\Xi\Vert_{(\gamma,\eta),1}<\infty.$ For the second part of the semi-norm $\Vert\Xi\Vert_{(\gamma,\eta)}$, we argue as follows. Firstly, for $\tau'\geq\tau$, we find
$$
\Xi^{\tau'}(t,s)-\Xi^{\tau}(t,s)=(K(\tau',s)-K(\tau,s))(W(t)-W(s))
$$
Hence, from the semi-norm in \eqref{Knorm 3}, we find
\begin{equation}
    \vert(K(\tau',s)-K(\tau,s))(W(t)-W(s))\vert_H\leq \|W\|_{\gamma}\Vert K\Vert_{\eta,3}\vert\tau'-\tau\vert^{\theta}\vert\tau-s\vert^{-\theta-\eta}|t-s|^{\gamma}.
\end{equation}
for any $\theta\in[0,1]$. Invoking \eqref{simple inequality}, it is readily seen that also $\|\Xi\|_{(\gamma,\eta),1,2}<\infty$.

Next we move on with 
showing the finiteness of $\vertiii{\delta\Xi}_{(\beta,\kappa)}$:
First, we investigate the action of $\delta$ on the integrand $\Xi$ given in \eqref{spec varxi}. By elementary algebraic manipulations, we observe that for $(\tau,t,u,s)\in \Delta_4^T$ the following relation holds
\begin{equation*}
    \delta_{u}\Xi^\tau(t,s)=\left(K(\tau,s)-K(\tau,u)\right)\left(W(t)-W(u)\right).
\end{equation*}
Again using that the kernel $K\in \cK_{\eta}$ and the assumption that $W\in \mathcal{C}^\gamma([0,T],H)$, it is readily checked that for any $\theta\in[0,1]$
\begin{equation*}
    \vert\delta_{u}\Xi^\tau(t,s)\vert_H\leq \|K\|_{\eta,3}\|W\|_\gamma |\tau-s|^{-\eta-\theta}\vert s-u\vert^{\theta}|t-u|^{\gamma}\lesssim \vert\tau-s\vert^{-\eta-\theta}\vert t-s\vert^{\gamma+\theta}.
\end{equation*}
We therefore set $\beta=\gamma+\theta$ and $\kappa=\eta+\theta$, and choose $\theta\in [0,1]$ such that $(\beta,\kappa)\in (1,\infty)\times(0,1)$ (we note that this is always possible due to the restriction $\gamma-\eta>0$).
Then we see that $\vert\tau-s\vert\geq\vert t-s\vert$ and $\vert\tau-s\vert\geq\vert\tau-t\vert$, and therefore
$$
\vert\tau-s\vert^{-\kappa}\vert t-s\vert^{\beta}\leq [ \vert\tau-t\vert^{-\kappa}\vert t-s\vert^{\beta}]\wedge\vert t-s\vert^{\beta-\kappa}
$$
It follows that $   \vertiii{\delta\Xi}_{(\beta,\kappa),1}<\infty$.
We also point out that $\beta-\kappa=\gamma-\eta$.

To prove that also $\vertiii{\delta\Xi}_{(\beta,\kappa),1,2}<\infty$, we follow in the same direction as outlined above. However, rather than invoking \eqref{bound2}, we will need to make use of \eqref{bound 3}. In particular, we need to consider the increment in the upper variables in  $\Xi$, i.e.
\begin{equation*}
    \Xi^{\tau',\tau}(t,s)=\Xi^{\tau'}(t,s)-\Xi^{\tau}(t,s),
\end{equation*}
and then the action of $\delta_u$ on $\Xi^{\tau',\tau}(t,s)$ is given by 
\begin{equation*}
    \delta_u\Xi^{\tau',\tau}(t,s)=\left(K(\tau',s)-K(\tau,s)-K(\tau',u)+K(\tau,u)\right)\left(W(t)-W(u)\right).
\end{equation*}
Thus invoking \eqref{bound 3} on the kernel $K$, and we can follow the exact same routine as for the proof that $\vertiii{\delta\Xi}_{(\beta,\kappa),1}<\infty$. One sees that for any parameters $\theta,\nu\in [0,1]$  and $(\tau',\tau,t,u,s)\in \Delta_5^T$  we have
\begin{equation*}
    \vert\delta_u\Xi^{\tau',\tau}(t,s)\vert_H \leq \|K\|_{\eta,4}\|W\|_\gamma |\tau'-\tau|^\nu |\tau-u|^{-\nu-\theta-\eta}|u-s|^\theta |t-u|^\gamma.
\end{equation*}
Using that for any $\zeta\geq 0$ we have $|\tau-u|^{-\nu-\theta-\eta}\leq |\tau-t|^{-\theta+\zeta}|\tau-u|^{-\eta-\nu-\zeta}$, we obtain that 
\begin{equation}
    \vert\delta_u\Xi^{\tau',\tau}(t,s)\vert_H \leq \|K\|_{\eta,4}\|W\|_\gamma |\tau'-\tau|^\nu|\tau-t|^{-\nu+\zeta}\left[|\tau-u|^{-\eta-\theta-\zeta}|t-s|^{\gamma+\theta}\right]
\end{equation}
We then choose $\theta\in [0,1]$ such that $\gamma+\theta>1$ and $\theta+\eta+\zeta<1$, which is possible by restricting $\zeta\in [0,\gamma-\rho)$, and $\gamma-\rho>0$ by assumption. We therefore set $\beta=\theta+\gamma$ and $\kappa=\theta+\eta$, and it follows that $\vertiii{ \delta \Xi}_{(\beta,\kappa),1,2}<\infty $, where we recall that this norm is defined in \eqref{dd3}. 
Thus we may invoke Lemma \ref{lem:(Volterra-sewing-lemma)} for the construction of the integral $\mathcal{I}(\Xi)$ as given in \ref{sum int spec}, and we get that this integral  exists with a unique limit. It follows directly from Lemma \ref{lem:(Volterra-sewing-lemma)}  that $X\in \mathcal{V}^{(\gamma,\eta)}(\Delta_2^T,H)$. This concludes the proof. 
\end{proof}

Let us illustrate Proposition \ref{regularity W} by providing an example which will be discussed in the applications, Section \ref{sect:applications}.

\begin{example}\label{example fbm}
For $(\tau,s)\in\Delta_2^T$, assume that the kernel $K(\tau,s)\in \mathcal{L}(H)$ is given on the form $K(\tau,s)=(\tau-s)^{-\eta}A$, where $\eta \in (0,\frac{1}{2})$ and  $A\in \mathcal{L}(H)$. Furthermore, for any $\alpha\in (0,\frac{1}{2})$ such that $\alpha>\eta$,  consider an $\alpha$-H\"older continuous trajectory of an $H$-valued $Q$-Wiener process
$W$.
Then we can give a pathwise construction of an infinite dimensional version of what is known as the Riemann-Liouville fractional Brownian motion by setting 
\begin{equation*}
    X^\tau (t)=\int_0^t(\tau-s)^{-\eta}AdW(s)=\mathcal{I}\left(\Xi^\tau\right)(t,0)
\end{equation*}
where the integral is constructed in terms of Proposition \ref{regularity W}.
An interesting observation here is that the construction of this processes is given as a purely deterministic functional $\mathscr{I}$ applied to the Wiener process $W$, i.e. $X=\mathscr{I}\left(W\right)$.
This tells us in particular that when we have constructed a Wiener processes on a probability space $(\Omega,\mathcal{F},\mathbb{P})$, and according to Kolmogorov's continuity theorem \ref{thm:Kolmogorov} found the set $\mathcal{N}^c\subset \Omega$ of full measure such that for each $\omega\in \mathcal{N}^c$ the mapping $t\mapsto W(\omega,t)$ has $\alpha-$H\"older continuous trajectories for $\alpha\in (0,\frac{1}{2})$, 
then the trajectory $(\tau,t)\mapsto X^\tau(\omega,t)\in\mathcal{V}^{(\alpha,\eta)}(\Delta_2^T,H)$.
Recall in particular from Proposition \ref{prop: Volterra implies Holder} that the restriction mapping  $t\mapsto X(\omega,t):=X^t(\omega,t)$ is $\rho$-H\"older continuous with $\rho=\alpha-\eta$.   This illustrates the point that simply from a probabilistic construction of the Wiener process, and the identification of the set of $\mathcal{N}^c\subset \Omega$ on which the Wiener process is continuous, one can construct a vast class of processes $X:\mathcal{N}^c\times[0,T]\rightarrow H$ given by $X(\omega,t)=\mathscr{I}(W(\omega,\cdot))(t)$. In the next section we will show that under mild conditions on the deterministic operators $K$, the random variable
\begin{equation*}
    \omega \mapsto X(\omega,t)=\mathscr{I}(W(\omega,\cdot))(t)
\end{equation*}
is  Gaussian on the probability space $(\Omega,\mathcal{F},\mathbb{P})$, with an explicit covariance operator given as a two-dimensional, possibly singular, integral with respect to the covariance operator of W. 
\end{example}

\section{Gaussian Volterra Processes}\label{sec: Gaussain Volterra processes}

With the Sewing Lemma \ref{lem:(Volterra-sewing-lemma)} at hand, we are now ready to investigate Volterra paths driven by Gaussian processes. The processes we consider will be constructed in a pathwise manner, as limits of 
Riemann-type  sums through the application  of Lemma \ref{lem:(Volterra-sewing-lemma)}. When the deterministic Volterra kernel $K$ is a linear operator on $H$ with sufficient regularity of the singularity,
we show that these processes are again Gaussian. More specifically, we consider Volterra processes on the form
\begin{equation}\label{eq:general Volt proc}
    X^\tau(t)=\int_0^t K(\tau,s)dW(s),
\end{equation}
where $\tau\geq t$ and  $W$ is a Gaussian  process with zero mean and a sufficiently regular covariance operator 
\begin{equation}\nonumber
Q_W(u,u'):=\mathbb{E}[W(u)\otimes W(u')]. 
\end{equation}
Recall from Section \ref{sec: inf dim gaussian analysis} that the covariance operator is a bounded linear operator on $H$. 
When the Volterra kernel $K\in \cK_{\eta}$  for some $\eta\in [0,1)$, and the covariance function $Q_W$ is sufficiently regular, we show that $X$ given in \eqref{eq:general Volt proc} is again a Gaussian process.
We derive the characteristic functional of $X$, and from this give an explicit computation of the covariance structure of $X$, denoted by $Q_X$.  
In fact, we show that the covariance operator $Q_X$ can be written as a deterministic functional of the kernel $K$  and the covariance of $W$. That is, the covariance operator  $Q_X$ can be written as 
\begin{equation}
\label{covar-functional-motivation}
Q_X=\mathscr{I}\left(K,Q_W\right),     
\end{equation}
where $\mathscr{I}$ is an integral operator, given as a double Young-Volterra integral. Furthermore, we prove that the operator $\mathscr{I}$ is Lipschitz continuous in both of its arguments. Stability of the covariance operator tells us in particular that if we do small (sufficiently regular) perturbations of the covariance associated to a Gaussian process $W$, then the covariance associated to the Gaussian process $X$ 
does not change by more than the size of these perturbations. In view of statistical estimation, this demonstrates robustness of the model with respect to data.

Let us begin to motivate the construction of the integral functional $\mathscr{I}$ in
\eqref{covar-functional-motivation}. The covariance operator $Q_X$ associated to $X$ will be defined by the double integral from $(0,0)$  to a point $(t,t') \in [0,T]^2$ as follows
\begin{equation}\label{eq:Qw integral rep}
    Q_X ^{\tau,\tau'}(t,t') = \int_0^t \int_0 ^{t'}K(\tau,r)d^2Q_W(r, r')K(\tau',r')^*,
\end{equation}
where   $K^*$ denotes the dual operator of $K$,  and the differential  $d^2Q_W$ will be given meaning below. 
If $Q_W$ is smooth, then we can think of this as given by the mixed partial derivative 
$d^2Q_W(r,r')= \frac{\partial^2Q_W}{\partial t\partial s}(r,r')drdr'$. 
From the proposed representation of $Q_W$, $Q_X^{\tau,\tau'}(t,t')f\in H$ since $K$ and $Q_W$ are both bounded linear operators on $H$.
At this stage, we would like to comment that the order of the integrands in \eqref{eq:Qw integral rep} is natural when working with operator-valued integrals corresponding to covariance functions. 
Since the covariance operators $Q_W$ and $K$ are linear operators on $H$, their non-commutative nature requires special care.  Indeed, first recall that for $(\tau,v,u), (\tau',v',u')\in \Delta_3^T$ and $f,g\in H$ we have 
\begin{equation*}
\mathbb{E}\left[ \langle K(\tau,u) W(u),f\rangle_H \langle K(\tau',v)W(v),g\rangle_H \right]
=\mathbb{E}\left[ \langle  W(u),K(\tau,u)^*f\rangle_H \langle W(v),K(\tau',v)^*g\rangle_H \right].
\end{equation*}
Since $X$ given in \eqref{eq:general Volt proc} is constructed as a limit of a Riemann sum as in Proposition \ref{regularity W}, let us motivate the construction of \eqref{eq:Qw integral rep} by considering an approximation of $X$ given by a partition $\mathcal P$ of $[0,t]$ as
\begin{equation}
    X^\tau_{\mathcal{P}}(t):=\sum_{[u,v]\in \mathcal{P}} K(\tau,u)( W(v)-W(u)).
\end{equation}
Then, the covariance operator between $X^\tau_{\mathcal{P}}(t)$ and $X^{\tau'}_{\mathcal{P}'}(t')$ (where $\mathcal{P}'$ is a partition of $[0,t']$) is computed in the following way
\begin{equation}\label{eq: co-variance comp}
\begin{aligned}
\mathbb{E}\bigg[& \langle \sum_{[u,v]\in \mathcal{P}} K(\tau,u)( W(v)-W(u)),f\rangle_H \langle \sum_{[u',v']\in \mathcal{P}'}K(\tau',u')(W(v')-W(u')),g\rangle_H \bigg] 
\\
&= \sum_{\substack{[u,v]\in \mathcal{P} \\ [u',v']\in \mathcal{P}'}}  \mathbb{E}\left[ \langle( W(v)-W(u)), K(\tau,u)^*f\rangle_H \langle (W(v')-W(u')),K(\tau',u')^*g\rangle_H \right] 
\\
&= \sum_{\substack{[u,v]\in \mathcal{P} \\ [u',v']\in \mathcal{P}'}} \left\langle \square_{(u,u'),(v,v')}Q_W K(\tau,u)^*f, K(\tau',u')^*g\right\rangle_H 
\\
&= \langle \sum_{\substack{[u,v]\in \mathcal{P} \\ [u',v']\in \mathcal{P}'}}  K(\tau',u') \square_{(u,u'),(v,v')}Q_W K(\tau,u)^*  f, g\rangle_H. 
\end{aligned}
\end{equation}
Here, we used the duality of linear operators and
\begin{equation}
     \square_{(u,u'),(v,v')}Q_W=
    \EE[(W(v)-W(u))\otimes(W(v')-W(u'))]
\end{equation}
by recalling the definition of the increment operator $\square$ in \eqref{eq:rec increment}. If $Q_W$ is mixed-differentiable in its two variables, 
we have
\begin{equation}
  \square_{(u,u'),(v,v')}Q_W\simeq \frac{\partial^2 Q_W}{\partial t\partial s}(u,v)(v-u)(v'-u'). 
\end{equation}
whenever $u$ is close to $u'$ and $v$ is close to $v'$.
However, we would like to allow for possibly singular covariance functions where the mixed partial derivative  $\frac{\partial^2Q_W}{\partial t\partial s} $ does not exist (possibly everywhere). Thus,  taking the limit when $|\mathcal{P}|\vee|\mathcal{P}'| \rightarrow 0$ in $X^{\tau}_{\mathcal{P}}$ and $X^{\tau'}_{\mathcal{P}'}$, one would need to show that the corresponding covariance integral appearing as the limit 
\begin{equation}
    \lim_{|\mathcal{P}|\vee|\mathcal{P}'| \rightarrow 0} \sum_{\substack{[u,v]\in \mathcal{P} \\ [u',v']\in \mathcal{P}}}  K(\tau',u') \square_{(u,u'),(v,v')}Q_W K(\tau,u)^*
\end{equation}
converges in $\cL(H)$.

\subsection{Construction of irregular covariance functions}

Our first goal will be to show the existence of the integral appearing on the right-hand side of \eqref{eq:Qw integral rep}. To this end, we will give an extension of the Volterra Sewing Lemma presented in Lemma \ref{lem:(Volterra-sewing-lemma)},  to allow for two-dimensional operator-valued integrals. As the integrals we are concerned with have the very specific form given in \eqref{eq:Qw integral rep}, we will tailor the construction of the two-dimensional integral to this specific case. 
Our second goal is to show that the process defined in \eqref{eq:general Volt proc} is Gaussian if $W$ is Gaussian and  $K$ is deterministic,  whenever the integral on the right-hand side of \eqref{eq:Qw integral rep} exists. 
Before moving on to the construction of the double integral in \eqref{eq:Qw integral rep}, we give a definition of a class of suitable two-parameter functions $Q$ which shall be used in the sequel for the construction of covariance operators.

\begin{defn}\label{def:reg covar}
Let $\alpha\in (0,1)$ and let  $Q:[0,T]^2\rightarrow \mathcal{L}(H)$. We say that $Q$ is an $\alpha$-regular covariance operator if it satisfies 
\begin{equation}\label{QNorm}
    \|Q\|_{\mathcal{Q}_\alpha}:= \|Q\|_{\alpha, (1,0)}+\|Q\|_{\alpha, (0,1)}+\|Q\|_{\alpha, (1,1)}<\infty,  
\end{equation}
where we define 
\begin{align}\label{Qnorm1}
    &\|Q\|_{\alpha, (1,0)}:=\sup_{\substack{(t,s)\in \Delta_2^T \\ t' \in [0,T]}} \frac{\|Q(t,t')-Q(s,t')\|_{op}}{|t-s|^\alpha}
    \\\label{Qnorm2}
    &\|Q\|_{\alpha, (0,1)}:=\sup_{\substack{t\in [0,T] \\ (t',s') \in \Delta_2^T}} \frac{\|Q(t,t')-Q(t,s')\|_{op}}{|t'-s'|^\alpha}
    \\\label{Qnorm3}
     & \|Q\|_{\alpha, (1,1)}:=\sup_{\substack{(t,s)\in \Delta^T_2 \\ (t',s') \in \Delta_2^T}} \frac{\|\square_{(s,s'),(t,t')}Q\|_{op}}{\left[|t-s||t'-s'|\right]^\alpha}, 
\end{align}
where we recall the rectangular increment is given by
\begin{equation}\nonumber
\square_{(u,u'),(v,v')}Q=Q(v,v')-Q(u,v')-Q(v,u')+Q(u,u'). 
\end{equation}
We denote the class of all $\alpha-$regular covariance operators by $\mathcal{Q}_\alpha$. 
\end{defn}
The reader should notice that the space $\mathcal Q_{\alpha}$
of $\alpha$-regular covariance operators is larger than 
the space of true covariance operators $Q:[0,T]^2\rightarrow\mathcal L(H)$ (with the same path-regularity, of course). Indeed, 
we have $Q(t,t)$ being a symmetric and positive-semidefinite
trace class operator if $Q(s,t)=\mathbb E[X(s)\otimes X(t)]$ is the covariance operator for a mean-zero and square-integrable $H$-valued stochastic process $X(t)$, a restriction not imposed on the elements in $\mathcal Q_{\alpha}$. Thus, our results in the next subsection 
cover a larger family of mappings $Q:[0,T]^2\rightarrow\mathcal L(H)$ than merely those which 
arise as covariance operators. We prefer to keep the adjective "covariance" associated to this larger class simply because 
we typically have such operators in mind.

\begin{rem}\label{rem: zero boundary co-variance}
If $Q:[0,T]^2\rightarrow \cL(H)$ is $0$ when one of the variables is $0$, i.e. $Q(0,t)=Q(t,0)=0$, then $Q\in \cQ_\alpha$ if $\|Q\|_{\alpha,(1,1)}<\infty$. Indeed, by subtraction of $0=Q(t,0)-Q(s,0)$ in \eqref{Qnorm1}, we observe that 
\begin{equation*}
    \|Q\|_{\alpha,(1,0)}=\sup_{\substack{(t,s)\in \Delta^T_2 \\ (t',s') \in \Delta_2^T}} \frac{\|Q(t,t')-Q(s,t')-Q(t,0)+Q(s,0)\|_{op}}{|t-s|^\alpha} \leq \|Q\|_{\alpha,(1,1)}T^\alpha. 
\end{equation*}
Similarly we can bound $\|Q\|_{\alpha,(0,1)}$. 
\end{rem}

The space $\cQ_\alpha$ is somewhat non-standard (at least from the point of view of covariance functions),  and thus we provide below an example of a co-variance operator contained in this space. 
We consider here the covariance operator of a fractional Brownian motion, and show that it is contained in such a space. For conciseness we only consider  the case of fractional Brownian motion with Hurst parameter $0<h\leq \frac{1}{2}$, as this is the case which will be discussed in later applications.

\begin{example}\label{co-variance of fBm}
Let $Q(t,s)=R^h(t,s)Q$ be the covariance operator of a fractional Brownian motion on a Hilbert space $H$ with Hurst parameter $h\in (0,\frac{1}{2}]$, where $R^h:[0,T]\rightarrow \RR$ is given as in \eqref{eq:fbm-r-func}. Then, we have that 
\begin{equation}
    \square_{(u,u'),(v,v')}R^h=\frac{1}{2}( -|v-v'|^{2h} +|v'-u|^{2h}+|v-u'|^{2h}-|u-u'|^{2h}).
\end{equation}
Using that for $\alpha\in (0,1]$, there exists a $c>0$ such that for two numbers $a,b\in \RR$,  $||a|^\alpha-|b|^\alpha|\leq c |a-b|^\alpha$, it follows that 
\begin{equation}
    |\square_{(u,u'),(v,v')}R^h|\leq c |v-u|^{2h} \wedge |v'-u'|^{2h}. 
\end{equation}
By using the interpolation inequality $a\wedge b\leq a^\theta b^{1-\theta}$ for any $\theta\in [0,1]$ and $a,b\in \RR_+$, we find that 
\begin{equation}
    |\square_{(u,u'),(v,v')}R^h|\leq c |v-u|^{h} |v'-u'|^{h}.
\end{equation}
It follows that the covariance operator $R^h(t,s)Q$ associated to a fractional Brownian motion with Hurst parameter $h\in (0,\frac{1}{2}]$ is contained in the space $\cQ_h$. 
\end{example}

The following theorem can be viewed as an extension (or combination) of the Volterra Sewing Lemma \ref{lem:(Volterra-sewing-lemma)} proven in \cite{HarTind} and the multi-parameter Sewing lemma found in \cite{Harang}.

\begin{thm}\label{thm:general covariance integrals}
Let $\alpha\in (0,1)$, $\eta\in [0,1)$ such that $\alpha-\eta>0$. Consider a covariance operator $Q:[0,T]^2 \rightarrow \cL(H)$ in $\cQ_\alpha$, 
and suppose  $K\in \mathcal{K}_{\eta}$ is 
a Volterra kernel. For partitions $\mathcal{P}$ of $[0,t]$ and $\mathcal{P}'$ of $[0,t']$, 
 define the approximating Volterra covariance function by 
\begin{equation}\label{partition M}
    M_{\mathcal{P}\times \mathcal{P}'}^{\tau,\tau'}(t,t'):=\sum_{\substack{[u,v]\in \mathcal{P} \\ [u',v']\in \mathcal{P}'}} K(\tau,u)\square_{(u,u'),(v,v')}Q K(\tau',u')^*.
\end{equation}
Then there exists a unique operator in $\cL(H)$ given as the  limit (in operator-norm)
\begin{equation}\label{integral def}
\mathcal{I}(K,Q)^{\tau,\tau'}(t,t'):=\int_{0}^t\int_{0}^{t'} K(\tau,r)d^2Q(r,r')K(\tau',r')^*:=\lim_{\substack{|\mathcal{P}|\rightarrow 0 \\ |\mathcal{P}'|\rightarrow 0}}M_{\mathcal{P}\times \mathcal{P}'}^{\tau,\tau'}(t,t'),
\end{equation}
satisfying the additivity relation 
\begin{equation}\label{additivity}
  \square_{(u,u'),(v,v')} \mathcal{I}(K,Q)^{\tau,\tau'}=\int_{u}^v\int_{u'}^{v'} K(\tau,r)d^2Q(r,r')K(\tau',r')^*
\end{equation} 
 Furthermore, there exists a pair $(\beta,\kappa)\in (1,\infty)\times[0,1)$ with $\beta-\kappa\geq \rho$ and a constant~$C>0$ such that the following statements holds
\begin{itemize}[leftmargin=0.8cm]
    \setlength\itemsep{.1in}
    \item[{\rm (i)}] For $(\tau,t,s),(\tau',t',s')\in \Delta_3^T$ the following inequality holds 
    \begin{multline}\label{reg of covar}
     \|\int_{s}^t \int_{s'}^{t'} \left(K(\tau,r)-K(\tau,s)\right)d^2Q(r,r')\left(K(\tau',r')^*-K(\tau',s')^*\right)\|_{op}
    \\
    \leq C \|K\|_{\mathcal{K}_{\eta}}^2\|Q\|_{\mathcal{Q}_\alpha} \left( \left[|\tau-t||\tau-t'|\right]^{-\kappa}\left[|t'-s'||t-s|\right]^{\beta} \right)\wedge \left[|\tau'-s'||\tau-s|\right]^{\beta-\kappa},
\end{multline}
\item[{\rm (ii)} ] For $(\tau,t,s)\in \Delta_3^T$, $(\tau_1',\tau_2',t',s')\in \Delta_4^T$ and any $\zeta\in [0,\rho)$ we have
\begin{multline}\label{reg of covar inc rec}
  \qquad  \|\int_{s}^t \int_{s'}^{t'} \left(K(\tau,r)-K(\tau,s)\right)d^2Q(r,r')\left(\square_{(\tau_2',s'),(\tau_1',r')}K^*\right)\|_{op}
    \leq C \|K\|_{\mathcal{K}_{\eta}}^2\|Q\|_{\mathcal{Q}_\alpha} |\tau_1'-\tau_2'|^\theta
    \\
    \times |\tau_2'-t'|^{-\theta+\zeta}\left( \left[|\tau-t||\tau-t'|^{-\zeta}\right]^{-\kappa}\left[|t'-s'||t-s|\right]^{\beta} \right)\wedge \left[|\tau-s||\tau'-s'|^{-\zeta}\right]^{\beta-\kappa}.
\end{multline}
\item[{\rm (iii)} ] For $(\tau_1,\tau_2,t,s)\in \Delta_4^T$,  $(\tau',t',s')\in \Delta_3^T $  and any $\zeta\in [0,\rho)$ we have
\begin{multline}\label{reg of covar rec inc}
  \qquad  \|\int_{s}^t \int_{s'}^{t'} \left(\square_{(\tau_2,s),(\tau_1,r)}K\right)d^2Q(r,r')\left(K(\tau',r')^*-K(\tau',s')^*\right)\|_{op}
    \leq C \|K\|_{\mathcal{K}_{\eta}}^2\|Q\|_{\mathcal{Q}_\alpha} |\tau_1-\tau_2|^\theta
    \\
    \times |\tau_2-t|^{-\theta+\zeta}\left( \left[|\tau-t||\tau-t'|^{-\zeta}\right]^{-\kappa}\left[|t'-s'||t-s|\right]^{\beta} \right)\wedge \left[|\tau-s||\tau'-s'|^{-\zeta}\right]^{\beta-\kappa}.
\end{multline}
\item[{\rm (iv)} ] For $(\tau_1,\tau_2,t,s),(\tau_1',\tau_2',t',s')\in \Delta_4^T$ and any $\zeta\in [0,\rho)$ we have
\begin{multline}\label{reg2 of covar}
  \qquad  \|\int_{s}^t \int_{s'}^{t'} \left(\square_{(\tau_2,s),(\tau_1,r)}K\right)d^2Q(r,r')\left(\square_{(\tau_2',s'),(\tau_1',r')}K^*\right)\|_{op}
    \leq C \|K\|_{\mathcal{K}_{\eta}}^2\|Q\|_{\mathcal{Q}_\alpha} [|\tau_1-\tau_2||\tau_1'-\tau_2'|]^\theta
    \\
    \times [|\tau_2-t||\tau_2'-t'|]^{-\theta+\zeta}\left( \left[|\tau-t||\tau-t'|^{-\zeta}\right]^{-\kappa}\left[|t'-s'||t-s|\right]^{\beta} \right)\wedge \left[|\tau-s||\tau'-s'|^{-\zeta}\right]^{\beta-\kappa}.
\end{multline}
\end{itemize}

\end{thm}

\begin{proof}
The first objective of this proof is to show the existence and uniqueness of the two-dimensional integral defined in \eqref{integral def}. To this end, we will also encounter one-dimensional integrals formed from the two-dimensional integrand $K(\tau,u)\square_{(u,u'),(v,v')}Q K(\tau',u')^*$ used in the definition \eqref{partition M} which will be called boundary integrals.  Since these integrals are simply constructed from the one-dimensional Volterra Sewing Lemma \ref{lem:(Volterra-sewing-lemma)} in the Banach space $\cL(H)$, we will only briefly comment on their construction here, and focus on the two-dimensional integral. The one-dimensional integrals are given on the form 
\begin{align}\label{N1}
   \mathcal{I}_{1}^{\tau,\tau'}(s,s',t,t') :=& \int_{s}^t K(\tau,r)\left[dQ(r,t')-dQ(r,s')\right]K(\tau',s)^*
   \\\nonumber
   =&\int_{s}^t\int_{s'}^{t'} K(\tau,r)d^2Q(r,r')K(\tau',s)^*,
\end{align}
where we note that there is no dependence on the integration variable  $r'$ in the integrand of the second integral, and thus the integral in this variable exists naturally as an integral over a constant. The differential $dQ(r,t')$ with fixed  second argument  $t'$ is then meant as the regular one-variable differential. We can define the second boundary integral $\mathcal{I}_2$ in the same way, by integrating over the interval $[s',t']$. 

In the sequel, we will frequently analyse the mapping
$$
(s,s',t,t',\tau,\tau')\mapsto
K(\tau,s)\square_{(s,s'),(t,t')}QK(\tau',s').
$$
From time to time, we will simply
write $K\square QK^*$ as a generic notation.

At this point we let $\delta^1$ and $\delta^2$ denote the $\delta$ given in \eqref{delta} restricted to the first and third,  and second and fourth variable, respectively, of a four-variable function $f(s,s',t,t')$. That is, the action on $\delta^1$ on $f$  for $s\leq u\leq t$  is given  by  
\begin{equation}\label{delta action}
    \delta^1_uf(s,s',t,t')=f(s,s',t,t')-f(u,s',t,t')-f(s,s',u,t'),
\end{equation}
and the action of $\delta^2$ is defined similarly over the variables $s' \leq u' \leq t'$. 
Then, using the Volterra Sewing Lemma \ref{lem:(Volterra-sewing-lemma)} we can show that there exists a pair $(\beta,\kappa)\in (1,\infty)\times[0,1)$ such that 
\begin{align}\label{1 dim reg N}
    \|\mathcal{I}_{1}^{\tau,\tau'}(s,s',t,t')-&K(\tau,s)\square_{(s,s'),(t,t')}Q K(\tau',s')^*\|_{op} 
    \\\nonumber
    \leq &C \vertiii{\delta^1 \left(K(\cdot,\cdot) \square_{(\cdot,s'),(\cdot,t')} QK(\tau',s')^*\right)}_{(\beta,\kappa),1}\left[|\tau-t|^{-\kappa}|t-s|^\beta\right]\wedge |\tau-s|^{\beta-\kappa},
\end{align}
Indeed, in order to apply Lemma \ref{lem:(Volterra-sewing-lemma)} we need to show that $\delta^1$ acting on the increment $K\square Q K^*$ is sufficiently regular. By elementary algebraic manipulations, we observe in particular that
\begin{align}\label{delta 1 on incr realtion}
     \delta^1_{z} K(\tau,u)&\square_{(u,s'),(v,t)}QK(\tau',s')^*
     =\left(K(\tau,u)-K(\tau,z)\right)\square_{(z,s'),(v,t)}Q K(\tau',s')^*.
\end{align}
With this relation at hand, invoking the assumption that  $K\in \mathcal{K}_{\eta}$ and the regularity condition on $Q$ given by \eqref{Qnorm1}, we obtain that for any $\theta\in [0,1]$
\begin{align}\label{1 dim delta bound}
    \| \delta^1_{z} K(\tau,u)\square_{(u,s'),(v,t')}QK(\tau',s')^*\|_{op}
    \leq \|K\|_{\eta,3} \|Q\|_{\alpha,(1,1)} \|K\|_{\eta,1} |\tau-z|^{-\eta-\theta}  |v-u|^{\alpha+\theta} T^{\alpha-\eta},
\end{align}
Here we have used that 
\begin{equation}
   \| \square_{(u,s'),(v,t')}Q\|_{op} \leq |v-u|^\alpha|t'-s'|^\alpha \qquad {\rm and}\qquad \|K(\tau',s')\|_{op}\leq |\tau'-s'|^{-\eta},
\end{equation}
 and thus 
\begin{equation}
    \| \square_{(u,s'),(v,t')}Q\|_{op}\|K(\tau',s')\|_{op}\leq |v-u|^\alpha T^{\alpha-\eta},
\end{equation}
where $\rho=\alpha-\eta$
 In the same way, 
one can verify a similar bound for~$\delta^2K\square Q K^*$. 
It follows that that the bound in~\eqref{1 dim reg N}
holds by setting $\beta=\alpha+\theta$ and $\kappa=\eta+\theta$ and choosing $\theta\in[0,1]$ such that $(\beta,\kappa)\in (1,\infty)\times[0,1)$ (which is always possible due to the fact that $\eta<\alpha$). The fact that $C$ in \eqref{1 dim reg N} can be chosen uniformly in all the time variables, follows from the assumption that $\eta<\alpha$, 
which in particular implies that any singularity coming from $K$ as $s' \rightarrow \tau'$ is killed by the regularity of $\square_{(\cdot,s'),(\cdot,t')} Q$, since $s' \leq t' \leq \tau'$. 
Similarly, we can define the integral $\mathcal{I}_2$ in the same way by letting the integrand be independent of the first integration variable $r$. By the same analysis as above, with application of the one-dimensional Volterra Sewing Lemma \ref{lem:(Volterra-sewing-lemma)}, one gets that also $\mathcal{I}_2$ is a well-defined integral, satisfying a similar bound to \eqref{1 dim reg N}.

Now let us focus on the two-dimensional integral operator $\mathcal{I}$ in \eqref{integral def}. First, we note that the  additivity relation in \eqref{additivity} is a straightforward consequence of the additivity of the limit of the one-dimensional Riemann sum, corresponding to the property
\begin{equation}\nonumber
 \int_0^t f(r)dr-\int_0^sf(r)dr=\int_s^t f(r)dr.    
\end{equation}
See \cite{FriHai} or \cite{Harang} for more details on this property in connection with the Sewing lemma both in the one-parameter and multi-parameter setting. 

For the {\it uniqueness} of the  integral defined in \eqref{integral def}, assume for now that the integral exists and satisfies {\rm (i)-(iv)}, and  consider the following argument: Assume  $\mathcal{M}$ and $\bar{\mathcal{M}}$ are two candidates for $\mathcal{I}$, both constructed from the integrand $K\square Q K^*$. 
We obtain for both $\mathcal M$ and $\bar{\mathcal M}$ one-dimensional integrals, which are unique by the one-dimensional sewing lemma. Thus, the boundary integrals $\mathcal I_1$ and $\mathcal I_2$ of both are identical. Note that this implies in particular that the boundary integrals corresponding to the difference $\mathcal{M}-\bar{\mathcal{M}}$ is equal to $0$. 
 Invoking this fact, it follows from  \eqref{reg of covar}  that the increment $\square_{(s,s'),(t,t')}(\mathcal{M}-\bar{\mathcal{M}})$ satisfies the following bound for $s\leq t\leq \tau$ and $s'\leq t'\leq \tau'$ and some $(\beta,\kappa)\in (1,\infty)\times [0,1)$
 \begin{equation}\label{bound diff M}
     \|\square_{(s,s'),(t,t')}(\mathcal{M}-\bar{\mathcal{M}})\|_{op} \leq C\left[|\tau-t||\tau-t'|\right]^{-\kappa}\left[|t'-s'||t-s|\right]^{\beta}. 
 \end{equation}
 Furthermore, due to the additive property in \eqref{additivity} of the increment we have that 
 \begin{equation}\label{sum exp}
     \mathcal{M}^{\tau,\tau \prime}(s,s',t,t')-\bar{\mathcal{M}}^{\tau,\tau \prime}(s,s', t,t')=\sum_{\substack{[u,v]\in \mathcal{P} \\ [u',v']\in \mathcal{P}'}} \square_{(u,u'),(v,v')}\left[\mathcal{M}^{\tau,\tau \prime}-\bar{\mathcal{M}}^{\tau,\tau \prime}\right],
 \end{equation}
 where $\mathcal{P}$ and $\mathcal{P}'$ are now partitions of $[s,t]$ and $[s',t']$ respectively. 
 Thanks to~\eqref{bound diff M} it follows that we can bound the left hand side of \eqref{sum exp} in the following way
 \begin{align*}
     \| \mathcal{M}^{\tau,\tau \prime}(0,0,t,t')-\bar{\mathcal{M}}^{\tau,\tau \prime}(0,0, t,t')\|_{op} &\leq C\sum_{\substack{[u,v]\in \mathcal{P} \\ [u',v']\in \mathcal{P}'}} \left[|\tau-v||\tau-v'|\right]^{-\kappa}\left[|v'-u'||v-u|\right]^{\beta}\\
    &\leq C\left[|\mathcal{P}||\mathcal{P}'|\right]^{\beta-1}\int_{s}^t\int_{s'}^{t'} \left[|\tau-r||\tau'-r' |\right]^{-\kappa} dr'dr,
 \end{align*}
 where we have appealed to the restriction of the parameters $(\beta,\kappa)\in  (1,\infty)\times [0,1)$ in the last inequality, as well as 
 recalling that $|\mathcal{P}|$ denotes the size of the mesh of the partition $\mathcal{P}$. We can now choose the partition arbitrarily fine, which  implies that the difference $\mathcal{M}^{\tau,\tau \prime}(0,0,t,t')-\bar{\mathcal{M}}^{\tau,\tau \prime}(0,0, t,t')\equiv 0$. We conclude that the integral constructed in \eqref{integral def} is unique. 

We continue with the proof of the {\it existence} of the integral $ \mathcal{I}(K,Q)$ given in \eqref{integral def}. To shorten slightly the notation, we from now on write $ \mathcal{I}:= \mathcal{I}(K,Q)$ for the two-dimensional integral.  In our argument, we first consider a sequence of approximating integrals constructed from dyadic partitions, and show that this sequence is Cauchy.  We show in particular that the integral constructed from dyadic partitions satisfy the regularity condition in {\rm (i)}. Second, we show that the definition may be extended to any partition, and thus the limit in Equation \eqref{integral def} is independent of the chosen partition~$\mathcal{P}$. 

Consider now dyadic partitions $\mathcal{P}^n$ for $n\in \mathbb{N}_0$ defined in the following way: $\mathcal{P}^0=\{[s,t]\}$  and  $\mathcal{P}^n$ is defined iteratively for $n\geq 1$ by 
\begin{equation}\label{dyadic part}
    \mathcal{P}^n:=\bigcup_{[u,v]\in \mathcal{P}^{n-1}} \{[u,z],[z,v]\},
\end{equation}
where the point $z=(v+u)/2$.
It is readily checked that each interval $[u,v]\in \mathcal{P}^n$ is of length $2^{-n}|t-s|$, and that $\mathcal{P}^n$ consists of $2^n$ intervals. Construct the dyadic partition $\mathcal{P}^{\prime,n'}$ of $[s',t']$ similarly for $n'\in\mathbb N_0$. 

For $n,n'\in \mathbb{N}$, observe that 
\begin{equation}
     M_{\mathcal{P}^n\times\mathcal{P}^{\prime,n'}}^{\tau,\tau'}-M_{\mathcal{P}^n\times\mathcal{P}^{\prime,n'-1}}^{\tau,\tau'} = \sum_{\substack{ [u,v]\in \mathcal{P}^{n} \\ [u',v']\in \mathcal{P}^{\prime,n'-1}}}\delta^2_{z'} K(\tau,u)\square_{(u,u'),(v,v')}Q K(\tau',u')^*. 
\end{equation}
From this we see 
that the following relation holds 
\begin{align}\label{M diff}
      M_{\mathcal{P}^n\times\mathcal{P}^{\prime,n'}}^{\tau,\tau'}-M_{\mathcal{P}^n\times\mathcal{P}^{\prime,n'-1}}^{\tau,\tau'}-&M_{\mathcal{P}^{n-1}\times\mathcal{P}^{\prime,n'}}^{\tau,\tau'} + M_{\mathcal{P}^{n-1}\times\mathcal{P}^{\prime,n'-1}}^{\tau,\tau'}
    \\\nonumber
   &= \sum_{\substack{ [u,v]\in \mathcal{P}^{n-1} \\ [u',v']\in \mathcal{P}^{\prime,n'-1}}}\delta^1_{z}\delta^2_{z'} K(\tau,u)\square_{(u,u'),(v,v')}Q K(\tau',u')^*.
\end{align}
Here $\delta^1\delta^2=\delta^1\circ \delta^2=\delta^2\circ \delta^1$ is the composition of the one-dimensional delta's given in \eqref{delta action}. We have already seen the action of  $\delta^i $ for $i=1,2$ applied to the increment~$K\square Q K^*$ in \eqref{delta 1 on incr realtion}, and we will therefore now compute the action of the composition $\delta^1\delta^2$. We get for $(v,z,u),(v',z',u')\in \Delta_3^T$ that 
\begin{equation}\label{eq:d1d2 relation}
    \delta^1_{z}\delta^2_{z'} K(\tau,u) \square_{(u,u'),(v,v')}Q K(\tau',u')^*
    =\left(K(\tau,u)-K(\tau,z)\right)\square_{(z,z'),(v,v')}Q \left(K(\tau',u')-K(\tau',z')\right)^*.
\end{equation}
Using the relation in \eqref{eq:d1d2 relation} together with the assumption that $K,K^*\in \mathcal{K}_{\eta}$ and that the covariance $Q$ is $\alpha$-regular, we obtain the following bound for any $\theta\in [0,1]$
\begin{align}\label{deltadelta bound}
    \| \delta^1_{z}\delta^2_{z'} K(\tau,u)\square_{(u,u'),(v,v')}&Q K(\tau',u')^*\|_{op}
    \\\nonumber
    &\leq \|K\|_{\eta,3}^2\|Q\|_{\alpha,(1,1)} \left[|\tau-z| |\tau'-z'|\right]^{-\eta-\theta}\left[ |v-u| |v'-u'|\right]^{\alpha+\theta},
\end{align}
where we have used that $|z-u|^\theta\leq |v-u|^\theta$, and similarly for the difference  $|z'-u'|^\theta$. 
With this inequality at hand, we will now go back to the difference in \eqref{M diff}. By telescoping sums, we observe that for $n'>m'$ we have 
\begin{equation}\label{incr}
M_{\mathcal{P}^n\times\mathcal{P}^{\prime,n'}}^{\tau,\tau'}-M_{\mathcal{P}^n\times\mathcal{P}^{\prime,m'}}^{\tau,\tau'}=\sum_{i=m'+1}^{n'} M_{\mathcal{P}^n\times\mathcal{P}^{\prime,i}}^{\tau,\tau'}-M_{\mathcal{P}^n\times\mathcal{P}^{\prime,i-1}}^{\tau,\tau'}, 
\end{equation}
with the same type of relation  for the difference $M_{\mathcal{P}^n\times\mathcal{P}^{\prime,n'}}^{\tau,\tau'}-M_{\mathcal{P}^m\times\mathcal{P}^{\prime,n'}}^{\tau,\tau'}$ when $n>m$. Combining the two and inserting the relation in \eqref{eq:d1d2 relation} yields that 
\begin{align}\label{M diff full}
     M_{\mathcal{P}^n\times\mathcal{P}^{\prime,n'}}^{\tau,\tau'}-M_{\mathcal{P}^n\times\mathcal{P}^{\prime,m'}}^{\tau,\tau'}-&M_{\mathcal{P}^m\times\mathcal{P}^{\prime,n'}}^{\tau,\tau'}+  M_{\mathcal{P}^m\times\mathcal{P}^{\prime,m'}}^{\tau,\tau'}
    \\\nonumber
   &=\sum_{\substack{i\in \{m+1,\ldots,n\} \\ j\in \{m'+1,\ldots,n'\}}} \sum_{\substack{[u,v]\in \mathcal{P}^i \\ [u',v'] \in \mathcal{P}^{\prime,j}}}    \delta^1_{z}\delta^2_{z'} K(\tau,z)\square_{(u,u'),(v,v')} Q K(\tau',z')^* .
\end{align}
Invoking the inequality obtained in \eqref{deltadelta bound}, we can bound the left-hand side of \eqref{M diff full} in the following way 
\begin{align}\label{eq:gen cauchy two d}
    \|  & M_{\mathcal{P}^n\times\mathcal{P}^{\prime,n'}}^{\tau,\tau'}-M_{\mathcal{P}^n\times\mathcal{P}^{\prime,m'}}^{\tau,\tau'}-M_{\mathcal{P}^m\times\mathcal{P}^{\prime,n'}}^{\tau,\tau'}+  M_{\mathcal{P}^m\times\mathcal{P}^{\prime,m'}}^{\tau,\tau'}\|_{op} 
    \\\nonumber
    &\leq \|K\|_{\eta,3}^2\|Q\|_{\alpha,(1,1)} \sum_{\substack{i\in \{m+1,\ldots,n\} \\ j\in \{m'+1,\ldots,n'\}}} \sum_{\substack{[u,v]\in \mathcal{P}^i \\ [u',v'] \in \mathcal{P}^{\prime,j}}} \left[|\tau-z| |\tau'-z'|\right]^{-\eta-\theta} \left[|v-u| |v'-u'|\right]^{\alpha+\theta}
    =: V_{\mathcal{P},\mathcal{P}'}. 
\end{align}
Then we choose $\kappa=\theta+\eta$ and $\beta=\theta+\alpha$ such that $(\beta,\kappa)\in (1,\infty)\times [0,1)$ (again we note that this is always possible since $\alpha-\eta>0$). Using that for any $[v,u]\in \mathcal{P}^n$ it holds $|v-u|=2^{-n}|t-s|$, we obtain that
\begin{align}\nonumber
    V_{\mathcal{P},\mathcal{P}'}&\leq C \left[|t-s||t'-s'|\right]^{\beta-1}\sum_{\substack{i\in \{m+1,\ldots,n\} \\ j\in \{m'+1,\ldots,n'\}}} 2^{-(i+j)(\beta-1)}\int_{s}^t\int_{s'}^{t'} \left[|\tau-r||\tau'-r'|\right]^{-\kappa}dr' dr
    \\\label{eq:two i j}
    & \leq C \left( \left[|\tau-t||\tau'-t'|\right]^{-\kappa}\left[|t-s||t' -s'|\right]^{\beta}\right) \wedge \left[|\tau-s||\tau'-s'|\right]^{\beta-\kappa} \sum_{\substack{i\in \{m+1,\ldots,n\} \\ j\in \{m'+1,\ldots,n'\}}} 2^{-(i+j)(\beta-1)}.
\end{align}
Inserting  \eqref{eq:two i j} into the right hand side of \eqref{eq:gen cauchy two d}, it follows that $\{M_{\mathcal{P}^n\times \mathcal{P}^{\prime, n \prime}}\}_{(n,n')\in \mathbb{N}^2}$ is a Cauchy sequence (with multi-index $(n,n')\in \mathbb{N}$).  We define the limit of this sequence as $n,n' \rightarrow \infty $ to be
\begin{equation*}
 \mathcal{I}^{\tau,\tau'}\left(s,s',t,t'\right):=\lim_{n,n'\rightarrow \infty } M_{\mathcal{P}^n\times\mathcal{P}^{\prime,n'}}^{\tau,\tau'}.
\end{equation*}
It now follows directly from the additivity property \eqref{additivity} proven in the beginning of the proof that the following identity holds
\begin{equation*}
     \mathcal{I}^{\tau,\tau'}\left(s,s',t,t'\right)=\square_{(s,s'),(t,t')} \mathcal{I}^{\tau,\tau'}(0,0,\cdot,\cdot).
\end{equation*}
Furthermore, due to the relations \eqref{1 dim delta bound} and \eqref{incr}, and by deriving a one-dimensional estimate similar to \eqref{eq:gen cauchy two d}, it follows that the boundary terms $\left\{M_{\mathcal{P}^n\times[s',t']}^{\tau,\tau'}\right\}_{n\in \mathbb{N}}$ and $\left\{M_{[s,t]\times\mathcal{P}^{\prime,n'}}^{\tau,\tau'}\right\}_{n' \in \mathbb{N}}$ are both Cauchy sequences. Moreover, we observe that  the boundary integrals are given as
\begin{align*}
     \mathcal{I}_{1}^{\tau,\tau'}(s,s',t,t')&=\lim_{n\rightarrow\infty}  M_{\mathcal{P}^n\times[s',t']}^{\tau,\tau'},
    \\
      \mathcal{I}_{2}^{\tau,\tau'}(s,s',t,t')&=\lim_{n'\rightarrow\infty}  M_{[s,t]\times\mathcal{P}^{\prime,n'}}^{\tau,\tau'}.
\end{align*}  
Note that these two objects are only additive in one pair of its variables, i.e. we have 
\begin{equation*}
 \mathcal{I}_{1}^{\tau,\tau'}(0,s',t,t')- \mathcal{I}_{1}^{\tau,\tau'}(0,s',s,t')= \mathcal{I}_{1}^{\tau,\tau'}(s,s',t,t')
\end{equation*}
while on the other hand 
\begin{equation*}
 \mathcal{I}_{1}^{\tau,\tau'}(s,0,t,t')- \mathcal{I}_{1}^{\tau,\tau'}(s,0,t,s') \neq  \mathcal{I}_{1}^{\tau,\tau'}(s,s',t,t').
\end{equation*}
The opposite relation holds for $ \mathcal{I}_2$. This is due to the nature of the integrand $K\square Q K^*$ and since 
\begin{equation*}
     \mathcal{I}_{1}^{\tau,\tau'}(s,s',t,t')=\lim_{n\rightarrow \infty}   M_{\mathcal{P}^n\times[s',t']}^{\tau,\tau'}=\lim_{n\rightarrow \infty} \sum_{[u,v]\in\mathcal{P}^n} K(\tau,u)\square_{u,s',v,t'}Q K(\tau',s')^*.
\end{equation*}
Recall that the boundary integrals $ \mathcal{I}_1$ should be thought of  as the integral appearing in \eqref{N1},
and similarly for $   \mathcal{I}_{2}$. 
Now,  we observe from \eqref{eq:gen cauchy two d}
with $m=m'=0$ that if we let 
\begin{equation*}
   \mathcal{H}^{\tau,\tau'}(s,s',t,t'):=
   \mathcal{I}^{\tau,\tau'}\left(s,s',t,t'\right)- \mathcal{I}_{1}^{\tau,\tau'}(s,s',t,t')
   - \mathcal{I}_{2}^{\tau,\tau'}(s,s',t,t')+K(\tau,s)\square_{(s,s'),(t,t')}Q K(\tau',s')^*
\end{equation*}
then we have 
\begin{multline}\label{4 comp ineq}
    \|  \mathcal{H}^{\tau,\tau'}(s,s',t,t')\|_{op} 
    \leq C\|K\|_{\eta,3}^2\|Q\|_{\alpha,(1,1)} 
    \\
   \left[\frac{1}{1-2^{1-\beta}}\right]^2 \left( \left[|\tau-t||\tau'-t'|\right]^{-\kappa}\left[|t-s||t' -s'|\right]^{\beta}\right) \wedge \left[|\tau-s||\tau'-s'|\right]^{\beta-\kappa},
\end{multline}
where $(\beta,\kappa)\in (1,\infty)\times[0,1)$ is chosen according to the rules specified below \eqref{eq:gen cauchy two d}. 
We conclude that the limiting objects $ \mathcal{I}$, $ \mathcal{I}_1$ and $ \mathcal{I}_2$ made from limits of dyadic partitions exist uniquely and satisfy the regularity condition in \eqref{reg of covar}.

The reader would note that we can construct the limiting object $ \mathcal{I}$ as a limit of a Riemann sum over either $ \mathcal{I}_1$ or $ \mathcal{I}_2$. By this we mean that the two-dimensional integral, can be obtained as a limit of a Riemann sum over the one-dimensional boundary integral. In particular we have that 
\begin{equation}\label{build M from N}
     \mathcal{I}^{\tau,\tau'}(s,s',t,t')=\lim_{n'\rightarrow \infty} \sum_{[u',v']\in \mathcal{P}^{\prime,n'}}  \mathcal{I}_1^{\tau,\tau'}(s,u',t,v'), 
\end{equation}
and similarly for $ \mathcal{I}_2$ where integration is done over a dyadic partition of $[s,t]$. 
Indeed, recall that $ \mathcal{I}_1$ is additive in the first variable, by which we mean that for any partition $\mathcal{P}$ we have 
 \begin{equation}\label{sum relation}
     \mathcal{I}_1^{\tau,\tau'}(s,s',t,t')=\sum_{[u,v]\in \mathcal{P}} \mathcal{I}_1^{\tau,\tau'}(u,s',v,t').
\end{equation}
A similar property holds for  $\cI_2$. Given the structure of the integrand $K\square QK^*$, and in the spirit of Lemma \ref{lem:(Volterra-sewing-lemma)} together with linearity of the integral it is readily checked that for any $\theta\in[0,1]$
\begin{equation}\label{N1 delta bound}
    \| \delta^2_{m'}  \mathcal{I}_{1}^{\tau,\tau'}(s,u',t,v') \|_{op}\leq C \|K\|_{\eta,1}\|Q\|_{\alpha,(0,1)}\|K\|_{\eta,2} [|\tau-m|^{-\eta-\theta}|v'-u'|^{\alpha+\theta}]|\tau-s|^{-\eta}|t-s|^\alpha,
\end{equation}
where the estimate is uniform in  the ordered variables $(s,t,\tau)$. Then again setting $\beta=\rho+\theta$ and $\kappa=\eta+\theta$ and choosing $\theta\in [0,1]$ such that $(\beta,\kappa)\in (1,\infty)\times[0,1)$, and next invoking \eqref{sum relation} together with \eqref{N1 delta bound}, it follows from the Volterra Sewing Lemma \ref{lem:(Volterra-sewing-lemma)} that 
\begin{multline}\label{M diff N1}
    \| \mathcal{I}^{\tau,\tau'}(s,s',t,t')-  \mathcal{I}_1^{\tau,\tau'}(s,s',t,t')\|_{op} \\ \leq C\|K\|_{\eta,1}\|Q\|_{\alpha,(0,1)}\|K\|_{\eta,2} |\tau-t'|^{-\kappa}|t'-s'|^\beta \wedge |\tau-s'|^{\beta-\kappa}]|\tau-s|^{-\eta}|t-s|^\alpha,    
\end{multline}
and thus relation \eqref{build M from N} follows directly.
Similarly, one can show that 
\begin{align*}
    \| \mathcal{I}^{\tau,\tau'}(s,s',t,t')&-  \mathcal{I}_2^{\tau,\tau'}(s,s',t,t')\|_{op} \\  &\leq C \|K\|_{\eta,2}\|Q\|_{\alpha,(1,0)}\|K\|_{\eta,1}|\tau-t|^{-\kappa}|t-s|^\beta \wedge |\tau-s|^{\beta-\kappa}]|\tau'-s'|^{-\eta}|t'-s'|^\alpha.    
\end{align*}

Our next goal is to show that the limiting object  $\cI$ is independent of the chosen partition $\mathcal{P}$. Note that the one-dimensional integral terms are in fact independent of the partition chosen, as a consequence of the one-dimensional Volterra Sewing Lemma \ref{lem:(Volterra-sewing-lemma)}.
Therefore, following from the relation \eqref{build M from N}, it is sufficient to show that the differences 
\begin{align}\label{M-N1 generic}
  & \mathcal{I}^{\tau,\tau'}(s,s',t,t') - \sum_{[u,v]\in \mathcal{P}^{\prime}}  \mathcal{I}_1^{\tau,\tau'}(s,u',t,v'),
  \\\label{M-N2 generic}
    & \mathcal{I}^{\tau,\tau'}(s,s',t,t') - \sum_{[u,v]\in \mathcal{P}}  \mathcal{I}_2^{\tau,\tau'}(u,s',v,t'),
\end{align}
converge to zero for generic partitions $\mathcal{P}'$ and $\mathcal{P}$, where $\vert\mathcal{P}'\vert \rightarrow 0$ and $\vert\mathcal{P}\vert \rightarrow 0$. Let us prove this for \eqref{M-N1 generic}. The same
result for \eqref{M-N2 generic} can be found by an analogous procedure. 
By additivity of $ \mathcal{I}$ and $ \mathcal{I}_1$, we can write 
\begin{equation}\label{M diff MPprime}
    \mathcal{I}^{\tau,\tau'}\left(s,s',t,t'\right)- \sum_{[u,v]\in \mathcal{P}^{\prime}}  \mathcal{I}_1^{\tau,\tau'}(s,u',t,v')= \sum_{[u,v]\in \mathcal{P}^{\prime}}  \mathcal{I}^{\tau,\tau'}\left(s,u',t,v'\right)-   \mathcal{I}_1^{\tau,\tau'}(s,u',t,v').
\end{equation}
Invoking the bounds we found in \eqref{M diff N1}, we can majorize the right-hand side of \eqref{M diff MPprime}, which yields that 
\begin{align}\nonumber
    \|  \mathcal{I}^{\tau,\tau'}\left(s,s',t,t'\right)- \sum_{[u,v]\in \mathcal{P}^{\prime}}  \mathcal{I}_1^{\tau,\tau'}(s,u',t,v')\|_{op} &\leq C \sum_{[u',v']\in \mathcal{P}' } |\tau'-v'|^{-\kappa}|u'-v'|^\beta 
    \\\nonumber
    &\leq C |\mathcal{P}' |^{\beta-1} \int_{s'}^{t'}|\tau'-r|^{-\kappa}dr,
\end{align} 
where the integral is convergent since $\kappa<1$, and  the constant $C>0$ may depend on $T^\rho$. Thus, letting $|\mathcal{P}'|\rightarrow 0$ we observe that 
\begin{equation*}
     \|  \mathcal{I}^{\tau,\tau'}\left(s,s',t,t'\right)- \sum_{[u,v]\in \mathcal{P}^{\prime}}  \mathcal{I}_1^{\tau,\tau'}(s,u',t,v')\|_{op}\rightarrow 0,
\end{equation*}
since $\beta>1$, and we conclude that the integral $ \mathcal{I}$ in \eqref{integral def} is independent of the choice of partition. We conclude that the limit in \eqref{integral def} exists uniquely, and it follows from \eqref{4 comp ineq} that the inequality in {\rm (i)} holds. 

It now remains to show that also {\rm (ii)-(iv)} holds. From the proof above, all the integrals appearing in these expressions exist, and so the different regularity estimates differ from {\rm (i)} in the sense that they have various increments in the upper parameters of the Volterra kernels.  As the proof of these inequalities are essentially identical with the proof of {\rm (i)} above, 
we will only show the inequality in {\rm (iv)} here, and leave the details for {\rm (ii)-(iii)} to the reader. This we do because {\rm (ii)-(iii)} can be seen as mixtures of {\rm (i)} and {\rm (iv)}, and it will therefore be simple to verify that also these inequalities hold. To illustrate this point, for $(\tau_1,\tau_2,t,s),(\tau_1',\tau_2',t',s')\in \Delta_4^T$  define $G^{\tau_1,\tau_2}(r)=K(\tau_1,r)-K(\tau_2,r)$, and observe that 
\begin{multline}
    \int_{s}^t \int_{s'}^{t'} \left(\square_{(\tau_2,s),(\tau_1,r)}K\right)d^2Q(r,r')\left(\square_{(\tau_2',s'),(\tau_1',r')}K^*\right) \\
    = \int_{s}^t \int_{s'}^{t'} \left(G^{\tau_1,\tau_2}(r)-G^{\tau_1,\tau_2}(s)\right)d^2Q(r,r')\left(G^{\tau_1',\tau_2'}(r')^*-G^{\tau_1',\tau_2'}(s')^*\right). 
\end{multline}
The right-hand side is an integral expression on the same form as in in {\rm (i)}, however with  different Volterra kernel. Similarly, we observe that {\rm (ii)} and {\rm (iii)} can be written as mixtures of integrals over the kernels $K$ and $G$ defined above. Following the strategy outlined above to prove {\rm (i)}, we now consider the integrand \begin{equation}
    G^{\tau_1,\tau_2}(s)\square_{(s,s'),(t,t')}Q G^{\tau_1',\tau_2'}(s')=(K(\tau_1,s)-K(\tau_2,s))\square_{(s,s'),(t,t')} Q (K(\tau_1',s')^*-K(\tau_2',s')^*),
\end{equation}
and by the same techniques as above our goal is to obtain an analytic inequality as in {\rm (iv)}. 
Consider the approximating integral given by 
\begin{equation*}
    N_{\cP\times \cP^{\prime}} := \sum_{\substack{[u,v]\in \cP \\ [u',v']\in \cP'}} (K(\tau_1,u)-K(\tau_2,u))\square_{(u,u'),(v,v')} Q (K(\tau_1',u')^*-K(\tau_2',u')^*).
\end{equation*}
Thus, we obtain the inequality in {\rm (iv)} by following the exact same steps as for the existence with the integrand $K\square Q K^*$. However, instead of relying on the norm $\|K\|_{\eta,3}$ as given in \eqref{Knorm 3} to obtain our bounds, we need to use $\|K\|_{\eta,4}$ given in \eqref{Knorm 4} as this represents the regularity of the kernel over the rectangular increment (i.e. in both upper and lower variables). Indeed, when arriving at the step similar to \eqref{eq:d1d2 relation}, we set 
\begin{equation*}
    \Xi^{\tau_1,\tau_2,\tau_1',\tau_2'}(u,u',v,v'):=(K(\tau_1,u)-K(\tau_2,u))\square_{(u,u'),(v,v')} Q (K(\tau_1',u')^*-K(\tau_2',u')^*) 
\end{equation*} and observe that
\begin{multline*}
    \delta^1_z\delta^2_{z'}  \Xi^{\tau_1,\tau_2,\tau_1',\tau_2'}(u,u',v,v')= (K(\tau_1,u)-K(\tau_2,u)-K(\tau_1,z)+K(\tau_2,z))
    \\
   \times \square_{(z,z'),(v,v')}  Q (K(\tau_1',u')^*-K(\tau_2',u')^*-K(\tau_1',z')^*+K(\tau_2',z')^*). 
\end{multline*}
We then need to bound this expression in a similar way as we did in \eqref{deltadelta bound}. Using the quantity defined in \eqref{Knorm 3}, it is readily seen that for $\theta_1,\theta_2\in [0,1]$
\begin{multline*}
    \| \delta^1_z\delta^2_{z'}  \Xi^{\tau_1,\tau_2,\tau_1',\tau_2'}(u,u',v,v')\|_{{\rm op}} \lesssim \|K\|_{\eta,4}^2\|Q\|_{\alpha,(1,1)}
    \\
    \times [|\tau_1-\tau_2||\tau_1'-\tau_2'|]^{\theta_1}[|\tau_2-v||\tau_2'-v'|]^{-\theta_1-\theta_2-\eta}|[|v-u||u'-v'|]^{\alpha+\theta_2}. 
\end{multline*}
We then observe that for a parameter $\zeta\in [0,1]$ we have 
\begin{equation*}
    |\tau_2-v|^{-\theta_1-\theta_2-\eta}\leq |\tau_2-t|^{-\theta_1+\zeta}|\tau_2-v|^{-\zeta-\theta_2-\eta}
\end{equation*}
and similarly for the parameters $(\tau_2',t',v')\in \Delta_3^T$. it follows that 
\begin{multline*}
    \| \delta^1_z\delta^2_{z'}  \Xi^{\tau_1,\tau_2,\tau_1',\tau_2'}(u,u',v,v')\|_{{\rm op}} \lesssim \|K\|_{\eta,4}^2\|Q\|_{\alpha,(1,1)}
    \\
    \times [|\tau_1-\tau_2||\tau_1'-\tau_2'|]^{\theta_1}[|\tau_2-t||\tau_2'-t'|]^{-\theta_1+\zeta}[|\tau_2-v||\tau_2'-v'|]^{-\theta_2-\eta-\zeta}|[|v-u||u'-v'|]^{\alpha+\theta_2}. 
\end{multline*}
Note that we are not integrating over the variables $(\tau_1,\tau_2,t),(\tau_1',\tau_2',t')\in \Delta_3^T$, and these will therefore not affect the sewing arguments in \eqref{eq:gen cauchy two d} and below. 
We now choose $\theta_1,\theta_2,\zeta\in [0,1]$ in the following way: $\beta=\alpha+\theta_2>1$ , $\kappa:=\eta+\theta_2+\zeta<1$. Since $\rho=\alpha-\eta>0$ we can choose $\zeta\in [0,\rho)$.  Then one can simply check that there exists a $\theta_2\in [0,1]$ such that $\beta>1$ and $\kappa<1$. By following the steps from \eqref{deltadelta bound} and below, one can conclude that {\rm (iv)} holds.  
This completes the proof. 
\end{proof}
\begin{rem}\label{rem: linearity of integral}
We point out that the integral $\int_0^t \int_0^{t'} K(\tau,r)d^2Q(r,r')K(\tau',r')^*$ is linear in $Q$, and bilinear in $K$. By this we mean that for $Q,\tilde{Q}\in \cQ_\alpha$
\begin{multline}
   \int_0^t \int_0^{t'} K(\tau,r)d^2[Q+\tilde{Q}](r,r')K(\tau',r')^*\\
   =\int_0^t \int_0^{t'} K(\tau,r)d^2Q(r,r')K(\tau',r')^*+\int_0^t \int_0^{t'} K(\tau,r)d^2\tilde{Q}(r,r')K(\tau',r')^*,
\end{multline}
and similarly for the bilinearity with respect to $K$. This follows directly from the construction of the integral as a limit of Riemann sums, and a simple verification can be done by going through the proof above using the integrand $K(\tau,u)\square_{(u,u'),(v,v')}[Q+\tilde{Q}]K(\tau',r')^*$. For conciseness we omit a more detailed proof here.  
\end{rem}

\begin{rem}
From the derivations in 
\eqref{eq: co-variance comp}, a different notation for the integral $\int_0^t \int_0^{t'} K(\tau,r)d^2Q(r,r')K(\tau',r')^*$ could be used. By
$$
\int_0^{t}K(\tau,r)\int_0^{t'}Q(dr,dr')K(\tau',r')^*
$$
we mean the integration of $K(\tau',r')^*$ with respect to $Q(r,dr')$ to form the integral
$\int_0^{t'}Q(r,dr')K(\tau',r')^*$, followed by
the integration of $K(\tau,r)$ with respect to the integrand $\int_0^{t'}Q(dr,dr')K(\tau',r')$. 
\end{rem}

A direct consequence of Theorem \ref{thm:general covariance integrals} is that the covariance functions constructed from $K\in \cK_{\eta}$ 
 and $Q\in \cQ_\alpha$ is again a covariance function in $\cQ_{\zeta}$ for any $\zeta\in [0,\alpha-\eta)$. We summarize this in the next Proposition. 

\begin{prop}\label{Holder reg of covar}
By restricting the domain of $\mathcal{I}$ to the square $[0,T]^2$ by considering the map $(t,t')\mapsto \mathcal{I}^{t,t'}(K,Q)(0,0,t,t')$, the integration map $\mathcal{I}$ is a continuous operator from $\mathcal{K}_\eta \times \mathcal{Q}_\alpha$  to $ \mathcal{Q}_{\zeta}$ for any $\zeta\in [0,\alpha-\eta)$. Moreover, we have that 
\begin{equation}\label{IQ mapto Q}
    \|\mathcal{I}(K,Q)\|_{\mathcal{Q}_{\zeta}}\leq C \|K\|^2_{\mathcal{K}_\eta} \|Q\|_{\mathcal{Q}_\alpha}. 
\end{equation}
\end{prop}

\begin{proof}
This follows from a combination of the estimates in {\rm(i)-(iv)} given in Theorem \ref{thm:general covariance integrals}. We denote by $\Delta_{s,t}K(\cdot,r)$ the increment $K(t,r)-K(s,r)$. Observe that
\begin{equation}\label{square inc exp}
\begin{aligned}
\square_{(s,s'),(t,t')} \int_{0}^\cdot \int_{0}^{\cdot'}& K(\cdot,r)  d^2Q(r,r')K(\cdot',r')^* = \int_s^t \int_{s'}^{t'}K(t,r)d^2Q(r,r')K(t',r')^*
\\
&+\int_0^s\int_{s'}^{t'}  \Delta_{s,t}K(\cdot,r)d^2Q(r,r')K(t',r')^* + \int_s^t\int_{0}^{s'}  K(t,r)d^2Q(r,r')\Delta_{s',t'}K(\cdot',r')^* 
\\
& + \int_0^s\int_{0}^{s'}  \Delta_{s,t}K(\cdot,r)d^2Q(r,r')\Delta_{s',t'}K(\cdot',r')^*. 
\end{aligned}
\end{equation}
Our goal is to check that 
\begin{equation}
    \|\square_{(s,s'),(t,t')} \int_{0}^\cdot \int_{0}^{\cdot'} K(\cdot,r)  d^2Q(r,r')K(\cdot',r')^*\|_{op} \lesssim [|t-s||t'-s'|]^{\alpha-\eta},
\end{equation}
and thus, by verifying that each of the integrals on the right-hand side of \eqref{square inc exp} satisfies the above bound, we are done. 
Each of the four terms on the right hand side above corresponds to the inequalities in {\rm (i)-(iv)} in Theorem \ref{thm:general covariance integrals} plus some one-dimensional integral terms which can be treated with the one-dimensional Volterra Sewing Lemma~\ref{lem:(Volterra-sewing-lemma)}. We will illustrate this by considering the first term on the right hand side of the above equality
\eqref{square inc exp}. It is readily checked that by addition and subtraction of three terms
$$
\int_s^t\int_{s'}^{t'}  K(t,s)d^2Q(r,r')K(t',r')^*, \,\,\, \int_s^t\int_{s'}^{t'}  K(t,r)d^2Q(r,r')K(t',s')^*,\,\,\, K(t,s)\square_{(s,s'),(t,t')}Q(r,r')K(t',s')^*,
$$
it follows that 
 \begin{align*}
      \int_s^t \int_{s'}^{t'}K(t,r)d^2Q(r,r')K(t',r')^*&= \int_s^t \int_{s'}^{t'}[K(t,r)-K(t,s)]d^2Q(r,r')[K(t',r')-K(t',s')]^*
     \\ 
     &\qquad+ \int_s^t \int_{s'}^{t'}K(t,s)d^2Q(r,r')K(t',r')^* \\
     &\qquad+\int_s^t \int_{s'}^{t'}K(t,r)d^2Q(r,r')K(t',s')^*-K(t,s)\square_{(s,s'),(t,t')} Q K(t',s'). 
 \end{align*}
 We can bound the first integral expression on the right-hand side by application of {\rm (i)} in Theorem \ref{thm:general covariance integrals}. The two other integral terms are one-dimensional in the sense that we are only integrating one of the kernels $K$ in either $r$ or $r'$. By the inequality obtained in \eqref{1 dim reg N} and  \eqref{1 dim delta bound}, it follows that 
 \begin{align}\label{eq:bound for 1 int 1}
     \|\int_s^t \int_{s'}^{t'}K(t,s)d^2Q(r,r')K(t',r')^*-K(t,s)&\square_{(s,s'),(t,t')} Q K(t',s')\|_{op} \nonumber \\ 
     &\qquad\lesssim \|K\|^2_{\cK_\eta}\|Q\|_{\alpha} |t-s|^{\alpha-\eta}|t'-s'|^{\alpha-\eta},
 \end{align}
 and similarly we get 
 \begin{align}\label{eq:bound for 1 int 2}
      \|\int_s^t \int_{s'}^{t'}K(t,r)d^2Q(r,r')K(t',s')^*-K(t,s)&\square_{(s,s'),(t,t')} Q K(t',s')\|_{op} \nonumber \\
      &\qquad\lesssim \|K\|^2_{\cK_\eta}\|Q\|_{\alpha} |t-s|^{\alpha-\eta}|t'-s'|^{\alpha-\eta}. 
 \end{align}
 At last, it is readily checked that also 
 \begin{equation}\label{eq: bound for 0 int}
     \|K(t,s)\square_{(s,s'),(t,t')} Q K(t',s')\|_{op} \lesssim \|K\|^2_{\cK_\eta}\|Q\|_{\alpha} |t-s|^{\alpha-\eta}|t'-s'|^{\alpha-\eta}.
 \end{equation}
 Thus a combination of \eqref{eq:bound for 1 int 1}, \eqref{eq:bound for 1 int 2} and \eqref{eq: bound for 0 int} as well as the bound in {\rm (i)} of Theorem \ref{thm:general covariance integrals}, we obtain \begin{equation*}
     \|\int_s^t \int_{s'}^{t'}K(t,r)d^2Q(r,r')K(t',r')^*\|_{op} \lesssim \|K\|^2_{\cK_\eta}\|Q\|_{\alpha} |t-s|^{\alpha-\eta}|t'-s'|^{\alpha-\eta}.
 \end{equation*}
By a similar analysis, one obtains equivalent bounds for the three other integral terms on the right-hand side of  \eqref{square inc exp} by appealing to {\rm (ii)-(iv)} of Theorem \ref{thm:general covariance integrals} as well as bounds for one-dimensional integral terms treated (as done above) by application of Lemma \ref{lem:(Volterra-sewing-lemma)}. However, in this case, the bound will be with respect to any exponent $\zeta\in [0,\alpha-\gamma)$, as the inequalities in {\rm (ii)-(iv)} satisfy this type of regularity condition. It therefore  follows that the left-hand side of \eqref{square inc exp} satisfies
\begin{equation*}
    \|\square_{(s,s'),(t,t')} \int_{0}^\cdot \int_{0}^{\cdot'} K(\cdot,r)  d^2Q(r,r')K(\cdot',r')^*\|_{op} \lesssim \|K\|^2_{\cK_\eta}\|Q\|_{\alpha} |t-s|^{\zeta}|t'-s'|^{\zeta},
\end{equation*}
for any $\zeta\in [0,\alpha-\gamma)$.
 Since the map $(t,t')\mapsto \int_{0}^t \int_{0}^{t'} K(\cdot,r)  d^2Q(r,r')K(\cdot',r')^*$ is zero on the boundary of $[0,T]^2$, we conclude by Remark \ref{rem: zero boundary co-variance} that the covariance operator is contained in $\cQ_{\zeta}$ for any $\zeta\in [0,\alpha-\gamma)$. 
\end{proof}

Another consequence of the construction of the double Young-Volterra integral is stability estimates in terms of the driving covariance $Q$ and the Volterra kernel $K$. We summarize this in the following proposition.

\begin{prop}
\label{prop:stability}
Let $K,\tilde{K}\in \mathcal{K}_{\eta}$,  
with $\eta\in(0,1)$.  For a constant $\alpha\in (\eta,1]$, assume that $Q$ and $\tilde{Q}$ are  both  $\alpha$-regular covariance functions in $\mathcal{Q}_\alpha$.
Furthermore,  let  $M>0 $ be a constant such that  $\|K\|_{\mathcal{K}_{\eta}}\vee \|\tilde{K}\|_{\mathcal{K}_{\eta}} \vee \|\mathcal{Q}\|_{\mathcal{Q}_\alpha}\vee \|\tilde{\mathcal{Q}}\|_{\mathcal{Q}_\alpha}\leq M$.
 Then the following stability estimate holds for any $\zeta\in [0,\alpha-\gamma)$.
\begin{equation}\label{eq:stability}
    \| \mathcal{I}(K,Q)- \mathcal{I}(\tilde{K},\tilde{Q})\|_{\mathcal{Q}_{\zeta}}\leq C_M\left(\|K-\tilde{K} \|_{\mathcal{K}_{\eta}} +\|Q-\tilde{Q}\|_{\mathcal{Q}_\alpha}\right).
\end{equation}
\end{prop}

\begin{proof}
This follows directly from the proof of  Theorem \ref{thm:general covariance integrals}, and Proposition \ref{Holder reg of covar}. First, it is readily checked that Theorem \ref{thm:general covariance integrals} may canonically be extended to integrals on the form
\begin{equation*}
    \mathcal{I}(K,Q,L)(t,t'):=\int_0 ^t \int_{0}^{t'}K(t,r)d^2Q(r,r')L(t',r')^*, 
\end{equation*}
where $K,L\in \mathcal{K}_{\eta}$ and $Q\in \mathcal{Q}_\alpha$. We therefore assume at this point that the above integral is well defined in the same way as shown in Theorem \ref{thm:general covariance integrals}. This leads to an extension of inequality \eqref{IQ mapto Q} given on the form 
\begin{equation}\label{int KQL}
    \| \mathcal{I}(K,Q,L)\|_{\mathcal{Q}_{\zeta}} \leq C \|K\|_{\mathcal{K}_\eta} \|Q\|_{\mathcal{Q}_\alpha} \|L\|_{\mathcal{K}_\eta},
\end{equation}
for $\zeta\in [0,\alpha-\gamma)$.
Observe that the difference $\mathcal{I}(K,Q)- \mathcal{I}(\tilde{K},\tilde{Q})$
is equal to 
 \begin{equation}\label{rel DKK DQQ}
     \mathcal{I}(K,Q)- \mathcal{I}(\tilde{K},\tilde{Q})=D(K,\tilde{K})+ D(Q,\tilde{Q})
 \end{equation}
 where we define
\begin{equation*}
D(K,\tilde{K}):=\mathcal{I}(K,Q)-\mathcal{I}(\tilde{K},Q) 
 \,\,\, {\rm and} \,\,\, D(Q,\tilde{Q}):=\mathcal{I}(\tilde{K},Q)-\mathcal{I}(\tilde{K},\tilde{Q}).
\end{equation*}
Recall from Remark \ref{rem: linearity of integral} that the integral operator is bilinear in $K$ and linear in $Q$. Moreover, since  $K$ and $\tilde{K}$ are both linear operators on $H$, their difference  is also a linear operator on $H$, and since $\mathcal{K}_\eta$ is a linear space,  it follows that $K-\tilde{K}\in \mathcal{K}_\eta$. Similarly, $Q-\tilde{Q}\in \cQ_{\alpha}$.  This yields,
\begin{align*}
    D(K,\tilde{K})(t,t')&=
    \int_0^t\int_{0}^{t'} \tilde{K}(t,r)d^2Q(r,r')(K-\tilde{K})(t',r')^*\\
   &\qquad+ \int_0^t\int_{0}^{t'} (K-\tilde{K})(t,r)d^2Q(r,r')K(t',r')^* \\
   &=\mathcal I(\tilde{K},Q,K-\tilde{K})(t,t')+\mathcal I(K-\tilde{K},Q,K)(t,t').
\end{align*}
Invoking the inequality   \eqref{int KQL} twice, we find 
\begin{equation}\label{DKK}
    \|D(K,\tilde{K})\|_{\mathcal{Q}_{\zeta}} \leq C_M \|K\|_{\mathcal{K}_\eta}\|K-\tilde{K}\|_{\mathcal{K}_\eta}\|Q\|_{\mathcal{Q}_\alpha}. 
\end{equation}
Through similar manipulations using that $Q-\tilde{Q}\in \mathcal{Q}_\alpha$ it is seen from Proposition \ref{Holder reg of covar} that $D(Q,\tilde{Q})$ can be bounded by
\begin{equation}\label{DQQ}
    \|D(Q,\tilde{Q})\|_{\mathcal{Q}_{\zeta}}\leq C_M \|K\|_{\mathcal{K}_\eta} \|Q-\tilde{Q}\|_{\mathcal{Q}_\alpha}. 
\end{equation}
We can now majorize the difference on the left hand side of \eqref{eq:stability} by using relation \eqref{rel DKK DQQ} and  the triangle inequality, as well as the estimates in \eqref{DKK} and \eqref{DQQ} to obtain 
\begin{equation*}
        \| \mathcal{I}(K,Q)- \mathcal{I}(\tilde{K},\tilde{Q})\|_{\mathcal{Q}_{\zeta}}\leq C_M\left(\|K-\tilde{K} \|_{\mathcal{K}_\eta} +\|Q-\tilde{Q}\|_{\mathcal{Q}_\alpha}\right),
\end{equation*}
which proves our claim. 
\end{proof}

The stability estimate in Proposition \ref{prop:stability} tells us that 
the Volterra processes are Lipschitz continuous in both the kernel $K$ and the covariance functional $Q$ of the noise. Thus, small model errors or statistical estimation errors in the kernel $K$ and/or the covariance functional $Q$ lead to small errors in the resulting Volterra processes. This holds $\omega$-wise and is therefore a very strong stability in a probabilistic context.

\subsection{Characteristic functionals of Volterra processes driven by Gaussian noise}
An important question to ask is whether the  pathwise Volterra process constructed in Proposition (\ref{regularity W}) is a Gaussian process when the driving noise $W$  is a Hilbert-valued Gaussian process.
The next proposition gives an affirmative answer to this question.

\begin{prop}\label{prop: pathwise process is gaussian }
Consider a Hilbert-valued zero-mean Gaussian process $W:[0,T]\times \Omega \rightarrow H$  with 
covariance operator  $Q_W:[0,T]^2\rightarrow \mathcal{L} (H)$,  and assume   $t\mapsto W(t,\omega)$ is $\beta$-H\"older continuous with $\beta\in (0,1)$ for  $\omega\in \cN^c\in\mathcal F$, where $\cN^c$ is of full measure.   
Let $K\in \cK_{\eta}$  with $\zeta:=\beta-\eta >0$, and that the covariance operator $Q_W\in \cQ_\alpha$ for $\rho=\alpha-\eta>0$. For any $\omega\in \cN^c$, let  $X^\cdot(\cdot,\omega)$ be given as the Volterra process
\begin{equation}\label{eq:general gaussian proc}
    X^\tau(t,\omega)=\int_{0}^{t}K(\tau,s)dW(s,\omega),
\end{equation} 
where the integral is constructed as in Proposition \ref{regularity W}. Then  $(t,\omega)\mapsto X^t(t,\omega)$ is a Hilbert-valued zero-mean Gaussian process on the probability space $\left(\Omega,\mathcal{F},\mathbb{P}\right)$, and  the characteristic functional of $X$ is given by 
\begin{equation}\label{char func X}
    \mathbb{E}\left[\exp\left(i\langle X^\tau(t),f\rangle \right)\right]=\exp\left(-\frac{1}{2}\langle \int_0^t\int_0 ^t K(\tau,r)d^2Q_W(r,r')K(\tau,r')^* f,f\rangle \right),
\end{equation}
for any $f\in H$.
\end{prop}
\begin{proof}
We begin to prove that for each $(\tau,t)\in \Delta_2^T$, $X^\tau(t,\cdot)$ is a Gaussian random variable. To this end,  it is sufficient to prove that the characteristic functional of $X^\tau(t,\cdot)$ is that of a Gaussian, and that it is given by \eqref{char func X}. Observe that by continuity of the exponential function and the construction of $X$ as the limit of a Riemann type sum as given in Proposition  \ref{regularity W}, we have 
\begin{align}\label{int to sum lim}
    \mathbb{E}\bigg[\exp\bigg(i\langle \int_{0}^{t}K(\tau,s)dW(s,\omega),f\rangle \bigg)\bigg]
    &=\mathbb{E}\bigg[\lim_{|\mathcal{P}|\rightarrow0}\exp\bigg(i\sum_{[u,v]\in \mathcal{P}} \langle K(\tau,u)(W(v)-W(u)),f\rangle \bigg)\bigg]
    \\\nonumber
    &=\mathbb{E}\bigg[\lim_{|\mathcal{P}|\rightarrow0}\exp\bigg(i\sum_{[u,v]\in \mathcal{P}} \langle W(v)-W(u),K(\tau,u)^*f\rangle \bigg)\bigg].
\end{align} 
Since the exponential $|\exp\left(i\langle g,f\rangle\right)|\leq 1$ for any $f,g\in H$, it follows from the  dominated convergence theorem that 
\begin{align}\label{lim outside}
    &\mathbb{E}\bigg[\lim_{|\mathcal{P}|\rightarrow 0}\exp\bigg(i\sum_{[u,v]\in \mathcal{P}} \langle W(v) -W(u),K(\tau,u)^*f\rangle \bigg)\bigg]
    \\\nonumber
    &\qquad\qquad= \lim_{|\mathcal{P}|\rightarrow 0}\mathbb{E}\bigg[\exp\bigg(i\sum_{[u,v]\in \mathcal{P}} \langle W(v)-W(u),K(\tau,u)^*f\rangle \bigg)\bigg].
\end{align} 
Using that the sum $\sum_{[u,v]\in \mathcal{P}} \langle W(v)-W(u),K(\tau,u)^* f\rangle$ is Gaussian, since $W$ is a Gaussian process (see Def. \ref{def:Hilbert Gaussian process}), and by similar computations as given in \eqref{eq: co-variance comp} we obtain that the following identity holds 
\begin{align}\label{car gaussian step}
    \mathbb{E}\bigg[\exp\bigg(i\sum_{[u,v]\in \mathcal{P}} \langle& W(v)-W(u),K(\tau,u)^*f\rangle \bigg)\bigg]
    \\\nonumber
    =&
    \exp\bigg(-\frac{1}{2}\sum_{\substack{[u,v]\in \mathcal{P}\\ [u',v']\in \mathcal{P}}} \langle \square_{(u,u'),(v,v')} Q_W K(\tau,u' )^*f,K(\tau,u)^*f\rangle \bigg).
\end{align}
By using the dual formulation of the operators again, and moving the double sum on this inside, we recognise that 
\begin{align*}
  \exp\bigg(-\frac{1}{2}&\sum_{\substack{[u,v]\in \mathcal{P}\\ [u',v']\in \mathcal{P}}} \langle \square_{(u,u'),(v,v')} Q_W K(\tau,u' )^* f,K(\tau,u)^*f\rangle \bigg)
  \\
  =&
  \exp\bigg(-\frac{1}{2} \langle \sum_{\substack{[u,v]\in \mathcal{P}\\ [u',v']\in \mathcal{P}}} K(\tau,u) \square_{(u,u'),(v,v')} Q_W K(\tau,u' )^*f,f\rangle \bigg).
\end{align*}
Taking limits as the partition goes to zero, we obtain exactly the operator-valued integral 
\begin{equation}\label{covar}
    \int_0^t\int_0 ^t K(\tau,r)d^2Q_W(r,r')K(\tau,r')^* =  \lim_{|\mathcal{P}|\rightarrow 0} \sum_{\substack{[u,v]\in \mathcal{P}\\ [u',v']\in \mathcal{P}}} K(\tau,u) \square_{(u,u'),(v,v')} Q_W K(\tau,u' )^*.
\end{equation}
By again recalling the derivations in \eqref{eq: co-variance comp}, we have that 
$$
\mathbb E\left[(X_{\mathcal P}^{\tau}(t))^{\otimes 2}\right]=
\sum_{\substack{[u,v]\in \mathcal{P}\\ [u',v']\in \mathcal{P}}} K(\tau,u) \square_{(u,u'),(v,v')} Q_W K(\tau,u' )^*.
$$
This shows that the right-hand side is symmetric and positive semi-definite operator, properties which are preserved after taking limits. Thus, the operator in \eqref{covar} a bounded linear operator on $H$ which is symmetric and positive semi-definite. 
Combining our considerations and identities obtained in \eqref{covar}, \eqref{car gaussian step}, \eqref{lim outside} and \eqref{int to sum lim},  we can see that 
\begin{equation}
     \mathbb{E}\left[\exp\left(i\langle \int_{0}^{t}K(\tau,s)dW(s),f\rangle \right)\right]=\exp\left(-\frac{1}{2}\langle \int_0^t\int_0 ^t K(\tau,r)d^2Q_W(r,r')K(\tau,r')^* f,f\rangle \right). 
\end{equation}
Recognising that this is the characteristic functional of a Gaussian random variable in a Hilbert space with trace class covariance operator  $Q^{\tau,\tau}_X(t,t)\in \mathcal{L}(H)$ given by  
\begin{equation}
    Q^{\tau,\tau}_X(t,t)=\int_0^t\int_0 ^{t} K(\tau,r)d^2Q_W(r,r')K(\tau,r')^*, 
\end{equation}
 proves that $X^\tau(t)=\int_0^t K(\tau,s)dW(s)$ is a Gaussian random variable in $H$ for each $(\tau,t)\in \Delta_2^T$. 
 
 In order to prove that $t\mapsto X(t)=X^t(t)=\int_0^tK(t,s)dW(s)$ is a Gaussian {\it process}, recall from Definition \ref{def:Hilbert Gaussian process}
that we need to show that for any $n\geq 1$,  $\{t_i\}_{i=1}^n\subset [0,T]$, and $\{f_i\}_{i=1}^n\in H^{\times n}$,  $(\la X(t_1),f_1\ra ,\ldots,\la X(t_n),f_n\ra )$ is an $n$-variate Gaussian random variable in $\RR^{ n}$. We prove this claim for $n=2$, and the case for $n\geq 2$ follows by by a similar argument, however being notationally much more involved.

For $t_1,t_2\in [0,T]$, we consider 
\begin{equation}\label{2 dim volterra process}
    \begin{pmatrix}
    \int_0^{t_1}K(t_1,r)dW(r)
    \\
    \int_0^{t_2}K(t_2,r)dW(r)
    \end{pmatrix} 
    \in H^2.
\end{equation}
Define an operator $G:[0,T]^4\rightarrow \cL(H^2)$ by 
\begin{equation}
    G(t_1,t_2,u_1,u_2)=
    \begin{pmatrix}
    K(t_1,u_1) & 0
    \\
    0 & K(t_2,u_2)
    \end{pmatrix}.
\end{equation}
Both integrals in \eqref{2 dim volterra process} are constructed as limits of Riemann type sums (as in Proposition \ref{regularity W}) in the following way: Set $\cP^1$ to be a partition over $[0,t_1]$ and $\cP^2$  to be a partition over $[0,t_2]$, and we have that 
\begin{equation}
    \begin{pmatrix}
    \int_0^{t_1}K(t_1,r)dW(r)
    \\
    \int_0^{t_1}K(t_2,r)dW(r)
    \end{pmatrix}
    =\lim_{|\cP^1|\rightarrow 0} \lim_{|\cP^2|\rightarrow 0} \sum_{[u_1,v_1]\in \cP^1 }
     \sum_{[u_2,v_2]\in \cP^2 }
    G(t_1,t_2,u_1,u_2)\begin{pmatrix}
    W(v_1)-W(u_1)
    \\
    W(v_2)-W(u_2)
    \end{pmatrix}
\end{equation}
Set $F=(f_1,f_2)\in H^2$, $u=(u_1,u_2),v=(v_1,v_2),t=(t_1,t_2)\in [0,T]^2$,  and define  
$$
Z(v)-Z(u):=\begin{pmatrix}
   W(v_1)-W(u_1)
    \\
    W(v_2)-W(u_2)
    \end{pmatrix}.
$$
It is then readily checked that  
\begin{multline}
    \EE \left[\la G(t,u) (Z(v)-Z(u)),F\ra_{H^2} \la G(t,u') (Z(v')-Z(u')),F\ra_{H^2}\right] 
    \\
    = \EE \left[\la  Z(v)-Z(u),G(t,u)^*F\ra_{H^2} \la  Z(v')-Z(u'),G(t,u')^*F\ra_{H^2}\right], 
\end{multline}
and by similar computations as in \eqref{eq: co-variance comp} we obtain the following expression
\begin{multline*}
     \EE \left[\la G(t,u) (Z(v)-Z(u)),F\ra_{H^2} \la G(t,u') (Z(v')-Z(u')),F\ra_{H^2}\right] 
     =\la G(t,u) \square_{(u,v),(u',v')}  Q_Z G(t,u')^* F, F\ra_{H^2} .
\end{multline*}
Let us first investigate the covariance $Q_Z$ associated to $Z$. By definition of $Z$, it follows that 
\begin{equation*}
\begin{aligned}
    \square_{(u,u'),(v,v')} & Q_Z  
    \\
    &=\begin{pmatrix}
    \EE[(W(v_1)-W(u_1))\otimes (W(v_1')-W(u_1'))] & \EE[(W(v_1)-W(u_1))\otimes (W(v_2')-W(u_2'))]
    \\
    \EE[(W(v_2)-W(u_2))\otimes (W(v_1')-W(u_1'))] & \EE[(W(v_2)-W(u_2))\otimes (W(v_2')-W(u_2'))]
\end{pmatrix}    
    \\ 
    &=\begin{pmatrix}
    \square_{(u_1,u_1'),(v_1,v_1')}Q_W & \square_{(u_1,u_2'),(v_1,v_2')}Q_W 
    \\
    \square_{(u_1',u_2),(v_1',v_2)}Q_W &\square_{(u_2,u_2'),(v_2,v_2')}Q_W  
    \end{pmatrix}.
\end{aligned}
\end{equation*}
The above expression for the covariance leads to the following expression for the appropriate composition of operators
\begin{multline*}
    G(t,u)\square_{(u,u'),(v,v')}Q_Z G(t,u')^* 
    \\
    = \begin{pmatrix}
    K(t_1,u_1)\square_{(u_1,u_1'),(v_1,v_1')}Q_W K(t_1,u_1')^* &K(t_1,u_1)\square_{(u_1,u_2'),(v_1,v_2')}Q_W K(t_2,u_2')^*
    \\
     K(t_2,u_2)\square_{(u_1',u_2),(v_1',v_2)}Q_W K(t_1,u_1')^* & K(t_2,u_2)\square_{(u_2,u_2'),(v_2,v_2')}Q_W K(t_2,u_2')^*
    \end{pmatrix}. 
\end{multline*}
The key observation here is that each of the elements in the above matrix only depends on four variables (in addition to $t_1$ and $t_2$). 
With this expression at hand, let $\cP:=\cP^1\times \cP^2$ and $\cP':=\cP^{\prime,1}\times \cP^{\prime,2}$ be two partitions of the rectangle $[0,t_1]\times[0,t_2]$. In particular, for $[u,v]=[u_1,v_1]\times[u_2,v_2]\in \cP$, $[u_1,v_1]\in \cP^1$ and $[u_2,v_2]\in \cP^2$. For notational ease define  $\sum_{\cP^i\times \cP^j}:=\sum_{[u_i,v_i]\in \cP^i}\sum_{[u_j,v_j]\in \cP^j}$ for $i,j=1,2$. We then have that  
\begin{multline*}
 \sum_{[u,v]\in \cP} \sum_{[u',v']\in \cP'} G(t,u) \square_{(u,v),(u',v')}  Q_Z G(t,u')^* 
\\
=  
\begin{pmatrix}
   \sum_{\cP^1\times \cP^{\prime,1}} K(t_1,u_1)\square_{(u_1,u_1'),(v_1,v_1')}Q_W K(t_1,u_1')^* &  \sum_{\cP^1\times \cP^{\prime,2}}K(t_1,u_1)\square_{(u_1,u_2'),(v_1,v_2')}Q_W K(t_2,u_2')^*
    \\
     \sum_{\cP^2\times \cP^{\prime,1}} K(t_2,u_2)\square_{(u_1',u_2),(v_1',v_2)}Q_W K(t_1,u_1')^* &  \sum_{\cP^2\times \cP^{\prime,2}} K(t_2,u_2)\square_{(u_2,u_2'),(v_2,v_2')}Q_W K(t_2,u_2')^*
    \end{pmatrix}
\end{multline*}
On the right-hand side we obtain four double-sums approximating different covariance operators, as constructed in Theorem \ref{thm:general covariance integrals}. In particular we have that for $i,j=1,2$
\begin{equation*}
    \lim_{\substack{|\cP^i|\rightarrow 0 \\ |\cP^j|\rightarrow 0}}  \sum_{\cP^i\times \cP^j} K(t_i,u_i)\square_{(u_i,u_j),(v_i,v_j)}Q_W K(t_j,u_j)^* = \int_0^{t_i}\int_0^{t_j} K(t_i,r)d^2Q_W(r,r')K(t_j,r')^*
\end{equation*}
from which we conclude that also the following expression is well-defined as a linear operator on $H^2$
\begin{equation*}
   \lim_{\substack{|\cP|\rightarrow 0 \\ |\cP'|\rightarrow 0}}  \sum_{[u,v]\in \cP} \sum_{[u',v']\in \cP'} G(t,u) \square_{u,v,u',v'}  Q_Z G(t,u')^* = \int_0^t\int_0^t G(t,s)d^2Q_Z(s,s')G(t,s')^*,  
\end{equation*}
where $|\cP|=|\cP^1|\vee |\cP^2|$, and similarly for $\cP'$. 
With all these tools at hand, we follow along the same lines of arguments leading to the proof that $\int_0^tK(t,s)dW(s)$ is a Gaussian random variable on $H$ as done in the first part of this proof, to see that  
\begin{equation*}
    \EE \left[ \exp(i\la \begin{pmatrix}
    \int_0^{t_1}K(t_1,r)dW(r)
    \\
    \int_0^{t_2}K(t_2,r)dW(r)
    \end{pmatrix},F \ra_{H^2})\right] =\exp\left( -\frac{1}{2} \la \int_0^t\int_0^t G(t,s)d^2Q_Z(s,s')G(t,s')^* F,F\ra_{H^2}\right), 
\end{equation*}
where $\int_0^t=\int_0^{t_1}\int_0^{t_2}$. From this it follows that  $\begin{pmatrix}
    \int_0^{t_1}K(t_1,r)dW(r)
    \\
    \int_0^{t_2}K(t_2,r)dW(r)
    \end{pmatrix}$ is a Gaussian random variable on $H^2$. A similar argument can be extended to any collection of times $t_1,\ldots,t_n\in [0,T]$,  and thus we conclude that  $t\mapsto \int_0^{t}K(t,r)dW(r)$ is a Gaussian process. 
\end{proof}
We remark in passing that the proof of Proposition \ref{prop: pathwise process is gaussian } shows more than only the covariance operator of $X^{\tau}(t)$. Indeed, the proof provides (by inductive arguments) the covariance operator associated with the $H^n$-valued random variable $(X(t_1),\ldots,X(t_n))$ for any sequence of times $\{t_i\}_{i=1}^n\subset[0,T]^n$, where $X(t):=X^t(t)$.

\section{Applications}\label{sect:applications}
In this Section we have collected some possible applications of our results on Gaussian Volterra processes in Hilbert space and the corresponding covariance functionals.
\subsection{Iterated stochastic process and their covariance operators}
 Iterated stochastic processes has received much attention (e.g. \cite{orsingher2009,Burdzy1993,BurdzyKhoshnevisan1998,ThiullenVigot2017}). In  \cite{BurdzyKhoshnevisan1998}, the authors propose to model a  diffusion  in a crack by iterated Brownian motions. In particular, one considers two independent Brownian motions $B^i:[0,T]\times \Omega_i \rightarrow \RR^n$ for $i=1,2$, and then studies properties of the process $B^1(|B^2(t)|)$.
  We refer to $B^1$ as the state process and $B^2$ as the time process.  Several interesting probabilistic and analytic  properties can be obtained from these processes, see in particular \cite{Burdzy1993} for a study of the pathwise properties of these processes, and \cite{orsingher2009} for relations with higher order fractional parabolic PDEs.
    A natural extension would be to consider infinite dimensional Gaussian processes indexed by irregular paths. 
    The advantage of this pathwise approach is that the time 
    process and the state process does not need to be independent. By this we mean that  we fix an $\omega_2\in \Omega_2$, such that $t\mapsto B^2(t,\omega_2)$ 
    is a continuous path, and one look at the conditional process $\BB(t)=B^1(|B^2(\omega_2,t)|)$ as a random variable.  This process is then a Gaussian process,
     and its covariance function is given by the composition of the  covariance function of $B^1$ with the path $|B^2(\omega_2,t)|$. Due to the fact that 
     $t\mapsto B^2(\omega_2,t)$ is H\"older continuous of order $\alpha<\frac{1}{2}$, it follows that the regularity of the covariance function is reduced accordingly. 
More generally, one can study infinite dimensional Gaussian processes with irregular time shifts. 
Let $I\subset \RR_+$, $\alpha\in (0,1)$, and  suppose $X:[0,T]\rightarrow I$, is a nowhere differentiable path, which is  $\alpha$-H\"older continuous. Let $W:I\times \Omega\rightarrow H$ be a Gaussian process with an $\gamma$-regular covariance function $Q_W:I\times I\rightarrow \cL(H)$ (according to Definition \ref{def:reg covar}). Then the composition $W\circ X:[0,T]\rightarrow H$ is a Gaussian process, with covariance function 
\begin{equation*}
    Q_W\circ X (t,s)(f,g)=\EE[\langle W\circ X(t),f\rangle \langle W\circ X(s),g\rangle ]= Q_W(X(t),X(s))\langle f,g \rangle.  
\end{equation*}
It follows that the covariance $Q_W\circ X$ is $\alpha \gamma$-regular with $\alpha\gamma\in(0,1)$.  Furthermore, one can  study the Volterra process $Y(t)=\int_0^t K(t,r)d(W\circ X) (r)$, in order to introduce memory in the iterated process. 
Then Proposition \ref{prop: pathwise process is gaussian } tells us that $Y$ is again Gaussian,  given that 
the singularity of $K$ is integrable with respect to the regularity of the covariance function $Q_W$. In fact, since the covariance operator $Q_W\circ X$ is only H\"older continuous, the covariance operator needs to be constructed in terms Theorem \ref{thm:general covariance integrals}, in order to make sense of this integral. This is of course due to the fact that 
$Q_W\circ X(t,s)$ is nowhere differentiable in a Fr\`{e}chet sense, and thus classical constructions of the covariance functions of Gaussian Volterra processes (for example given in \cite{HuCam}) are not applicable. 

\subsection{Construction of the rough path lift of Gaussian processes with irregular covariance functions}
 At the core of rough paths lies a solution theory for controlled differential equations on the form 
\begin{equation*}
dY(t)=f(Y(t))dX(t),\qquad Y(0)=y\in H,
\end{equation*}
where $f$ is sufficiently regular function and $X$ is an $\alpha$-H\"older continuous signal with $\alpha\in (0,1)$. If $\frac{1}{3}<\alpha\leq \frac{1}{2}$, one needs to lift the signal $X$ into a tuple $(X,\XX)$, where $\XX:[0,T]^2\rightarrow H \otimes H$ represents the iterated integral of $X$. This tuple is then called the rough path corresponding to $X$.  In fact, one requires the following two conditions to hold for $X$ and $\XX$ for $s\leq u\leq t$,
\begin{align*}
\XX(s,t)-\XX(s,u)-\XX(u,t)&=(X(u)-X(s))\otimes (X(t)-X(u)),
\end{align*}
and
\begin{align*}
\sup_{t\neq s\in [0,T]} \frac{|X(t)-X(s)|_H}{|t-s|^\alpha}<\infty \qquad &{\rm and} \qquad \sup_{t\neq s\in [0,T]} \frac{|\XX(s,t)|_H}{|t-s|^{2\alpha}}<\infty.
\end{align*}
Therefore, much attention is given to construct an object $\XX$ which satisfies the above conditions for a given path $X$. In \cite[Sec. 10.2]{FriHai}, the construction of this object corresponding to a Gaussian noise is shown under a sufficient smoothness condition on the covariance function. This smoothness condition is stated in terms of two-dimensional $p$-variation norms, which can be seen to be equivalent to the H\"older continuity of the covariance operators introduced in Definition \ref{def:reg covar} under the assumption of continuity on the $p$-variation functions

In particular, in order to construct a "geometric" version of $\XX$ when $X$ is a centred Gaussian process,  \cite[Thm. 10.4]{FriHai} tells us that it is sufficient that the covariance operator $Q_X$ is contained in $\cQ_\gamma$ with $\gamma>\frac{1}{2}$
\footnote{The condition is actually stated in terms of a two-dimensional $\rho$-variation norm for the covariance function, with~$\rho\in [1,2)$. It is readily checked that the H\"older norms in Definition \ref{def:reg covar} are equivalent to this variation norm, under the assumption of continuity.}. Thus, the construction of covariance operators and their corresponding regularity provided in Theorem \ref{thm:general covariance integrals} and Proposition \ref{Holder reg of covar}, open up for the construction of a rough path for Volterra processes driven by Gaussian paths with nowhere differentiable covariance  operators. Such processes are, for example, illustrated in the above subsection by the class of iterated processes.

\subsection{Fractional Ornstein-Uhlenbeck process driven by irregular paths}

Fractional differential equations (FDEs) provide an alternative to classical ODEs, by introducing memory in the evolution of the process. This results in a non-local equation with interesting applications to several physical and social systems (e.g. \cite{FracDiff2019,EuchRosen,SINGH2018}) 


Our concern here is an $H$-valued fractional Ornstein-Uhlenbeck stochastic differential equation on a given time interval $[0,T]$. Consider two parameters $(\alpha,\gamma)\in \mathbb{R}_+\times(0,1)$ with the relation $\alpha+\gamma-1>0$, and consider the equation  formally given by  
\begin{equation}\label{FDE}
D^{\alpha}\left(Y-y\right)(t)=AY(t)+\dot{W}(t).
\end{equation}
Here, $y\in H$, $A\in\mathcal L(H)$, $W\in\mathcal{C}^{\gamma}([0,T],H)$ and $D^{\alpha}$ is the fractional time-derivative of order $\alpha$, given as in Definition \ref{fractional integral and derivative} in the Appendix. The object $\dot{W}$ is interpreted only formally  and is corresponding to the time-derivative $\frac{d}{dt}W(t)$.
Since $W$  is only H\"older continuous, the derivative  $\frac{d}{dt}W(t)$  does not exist, and thus we rather consider an integrated version of~\eqref{FDE}. 
With $I^{\alpha}$ being the fractional integral operator
(see Definition \ref{fractional integral and derivative} in the appendix), let us denote by $X=I^\alpha(\dot{W})$ which we interpret as the integral
\begin{equation}\label{X=I(W)}
    X(t)=\int_0^t(t-s)^{\alpha-1}dW(s).
\end{equation}
This integral is understood in the sense of Proposition \ref{regularity W} with $K(t,s)=(t-s)^{\alpha-1}I$ and~$I\in \mathcal{L}(H)$ being the identity operator on $H$. The integral exists due to the assumption that  $\alpha+\gamma-1>0$.
Applying the fractional integral operator $I^\alpha $ on both sides of~\eqref{FDE}, we obtain the equation 
\begin{equation}\label{mildFDE}
Y(t)=y+I^{\alpha}\left(AY\right)(t)+X(t), \qquad t\in[0,T].
\end{equation}
We will need a few extra tools to be able to obtain an explicit representation of its solution, as well as the associated covariance operator in the Gaussian case.

First we present a version of Fubini's theorem, showing that we can exchange the order of integration of double integrals involving Riemann integration and Volterra-Young integration. This property, that may be interesting in itself, is a crucial tool in proving a specific analytic representation of the fractional Ornstein-Uhlenbeck process in \eqref{mildFDE}. As a corollary
to our Fubini theorem, we show that the order of the  fractional integral operator and a Young-Volterra integral can be interchanged. For conciseness, all proofs in this section are relegated to Appendix \ref{A: auxiliary proofs}.

\begin{prop}\label{Fubinis prop}
For $\gamma,\eta>0$  with $\rho:=\gamma-\eta>0$, let $Z:[0,T]\rightarrow H$  be given as 
\begin{equation*}
    Z(t)=\int_0 ^t K(t,s)dW(s),
\end{equation*}
for $K\in \mathcal{K}_\eta$ and $W\in \mathcal{C}^\gamma([0,T],H)$, with the integral being defined as in Proposition \ref{regularity W}. Assume $G:\Delta_2^T\rightarrow \mathcal{L}(H)$ is in $\mathcal{K}_\kappa$ for some $\kappa\in (0,1)$. Then the following equality holds
\begin{equation}\label{fubini}
    \int_{0}^{t}G(t,s)Z(s)ds=\int_{0}^{t}\int_{s}^{t}G(t,r)K(r,s)drdW(s), 
\end{equation}
where the integral on the right-hand side is again interpreted in terms of Lemma \ref{lem:(Volterra-sewing-lemma)} with $\Xi^{\tau}(t,s):=\int_{s}^{\tau}G(\tau,r)K(r,s)dr\left(W(t)-W(s)\right)$. 
\end{prop}

As already indicated, we apply the Fubini theorem to the fractional integral operator, and as we see in ther next Corollary, we can further establish a connection to Mittag-Leffler functions.

\begin{cor}
\label{cor:Fubini with cont path}Let $0<\alpha<\gamma<1$   and $\alpha+\gamma-1>0$ and $W\in\mathcal{C}^{\gamma}\left([0,T],H\right)$.  Let furthermore $X$  be defined as in Proposition \ref{regularity W}, with $K(t,s)=\frac{1}{\Gamma(1-\eta)}(t-s)^{-\eta}$  
and $\gamma-\eta>0$. In particular, $X$ is given as the Volterra integral
\begin{equation}\nonumber
    X(t)=\Gamma(1-\eta)^{-1}\int_0^t (t-s)^{-\eta}dW(s).
\end{equation}
 Then, for $A\in \mathcal{L}(H)$, the following relation holds
\begin{equation}
\label{infinite sume mittag lef}
\sum_{i=0}^{\infty}A^{\circ i}I^{i\alpha}\left( X\right)(t)=\int_{0}^{t}\left(t-s\right)^{-\eta}E_{\alpha,1-\eta}\left(A \left(t-s\right)^{\alpha}\right)dW(s),
\end{equation}
where $E_{\alpha,\beta}\left(At\right):=\sum_{i=0}^{\infty}\frac{A^{\circ i}t^{i}}{\Gamma\left(\alpha i+\beta\right)}$ for $t\in [0,T]$ is called the Mittag-Leffler operator, and the integrals are interpreted in sense of Proposition \ref{regularity W}. Indeed, since $x\mapsto E_{\alpha,\beta}(x^\alpha)$ is smooth everywhere except at $0$ where it is $\alpha$-H\"older continuous, we interpret the right-hand side of \eqref{infinite sume mittag lef} using Proposition \ref{regularity W} with $K(t,s)=(t-s)^{-\eta}E_{\alpha,1-\eta}(A(t-s)^\alpha)$. 
\end{cor}

\begin{rem}
The fact that the Mittag-Leffler operator  is a bounded linear operator on $H$ is readily checked: for any $f\in H$ we have by the triangle inequality that 
\begin{equation}\label{tripple ineq}
    |E_{\alpha,\beta}(At)f|_H=|\sum_{i=0}^{\infty}\frac{A^{\circ i}t^if}{\Gamma\left(\alpha i+\beta\right)}|_H\leq \sum_{i=0}^{\infty}\frac{|A^{\circ i}t^i f|_H}{\Gamma\left(\alpha i+\beta\right)}\leq |f|_H E_{\alpha,\beta}(\|A\|_{op}t),
\end{equation}
where, in the last inequality, we have used that for a bounded linear operator $A$ and for any $i\geq 0$ we have   $|A^{\circ i}f|_H\leq \|A\|_{op}^i |f|_H$. The expression $E_{\alpha,\beta}(\|A\|_{op}t)$ appearing on the right-hand side of \eqref{tripple ineq} is the classical Mittag-Leffler function evaluated at $\Vert A\Vert_{op}t$. 
\end{rem}

\begin{thm}\label{thm:FDE}
For some $\gamma\in\left(0,1\right)$, let $W\in\mathcal{C}^{\gamma}\left(\left[0,T\right],H\right)$,  and assume that $A\in \mathcal{L}(H)$.
For any $\alpha>1-\gamma$ let $X=I^\alpha (W)$ as given in \eqref{X=I(W)}, and assume   $y\in H$. Then  there exists a unique solution $Y\in\mathcal{C}^{\rho}([0,T],H)$ with $\rho<\gamma+\alpha-1 $ to the equation 
\begin{equation}
Y(t)=y+AI^{\alpha}\left(Y\right)(t)+X(t).\label{eq: fractional equation}
\end{equation}
Moreover, the solution satisfies the following analytic formula 
\begin{equation}\label{eq: representation mittag lefler}
Y(t)=E_{\alpha,1}\left(At^{\alpha}\right)y+\int_{0}^{t}\left(t-s\right)^{\alpha-1}E_{\alpha,\alpha}\left(A\left(t-s\right)^{\alpha}\right)dW(s),
\end{equation}
where the integral on the right-hand side of \eqref{eq: representation mittag lefler} is interpreted in sense of Corollary \ref{cor:Fubini with cont path}. 

\end{thm}
We observe from our analysis in Section \ref{sec: inf dim gaussian analysis} and \ref{sec: Gaussain Volterra processes}, that $Y$ is a Gaussian process. In the next Corollary we apply Theorem \ref{thm:general covariance integrals} to state the covariance operator of $Y$.
\begin{cor}
\label{cor:ML-covariance}
Consider parameters $\gamma,\alpha\in(0,1)$ such that $\rho=\gamma+\alpha-1>0$ and $\beta>0$ such that $\beta+\alpha-1>0$. Let $W$ be a Gaussian process in $\cC^\gamma([0,T],H)$, with covariance operator $Q_W\in \cQ_\beta$, and suppose $Y\in \cC^\rho([0,T],H)$ is the solution to the fractional Ornstein-Uhlenbeck process given in Theorem \ref{thm:FDE} driven by $W$ with linear operator $A\in \cL(H)$. Then the covariance operator associated to $Y$ is given by 
\begin{equation}
    Q_Y(t,t')=\int_{0}^t\int_0^{t'}(t-r)^{\alpha-1}E_{\alpha,\alpha}(A(t-r)^\alpha) d^2Q_W(r,r')(t'-r')^{\alpha-1}E_{\alpha,\alpha}(A^*(t'-r')^\alpha) . 
\end{equation}
and $Q_Y\in \cQ_{\eta}$ for any $\eta<\beta+\alpha-1$. 
\end{cor}

\begin{rem}
Observe that the regularity of $Y$ constructed as the Volterra process in \eqref{eq: representation mittag lefler} is of order $0<\rho<\gamma+\alpha-1$, where $\gamma$ is the regularity of $W$,   However the regularity of the covariance  $Q_Y$ is of order $\eta<\beta+\alpha-1$ where $\beta$ is the regularity of the covariance $Q_W$. A-priori, there is no imposed relationship between $\beta$ and $\gamma$, although in typical examples they will be strongly related (if not the same, see Example \ref{example fbm} for the case of fractional Brownian motion). On the other hand, given that we know the regularity of the covariance operator $Q_W$, then through Kolmogorov's continuity theorem \ref{thm:Kolmogorov}, one can deduce the regularity of $W$ (which then relates $\gamma$ to $\beta$). However,  this theorem is not an if and only if statement, and thus given that a stochastic process is $\gamma$-H\"older continuous, it is not obvious what regularity its covariance might have. 
\end{rem}

\subsection{Rough stochastic volatility models}
In this subsection we discuss various infinite dimensional extensions of {\it rough} stochastic volatility models that have attracted interest in recent years. Our starting point is the
fractional Ornstein-Uhlenbeck process $Y$ defined in~\eqref{eq: representation mittag lefler}, where we for simplicity assume $y=0$. 

Consider first a state space $H=\mathbb R$, and let the risk-neutral stock price dynamics with stochastic volatility be 
$$
\frac{dS(t)}{S(t)}=\sigma(t,Y(t))dB(t),
$$
for some Brownian motion $B$, possibly correlated with $W$, and where we suppose the risk-free interest rate to be zero. Recall that $W$ is the Gaussian process driving the fractional Ornstein-Uhlenbeck dynamics of $Y$.  For example, choosing $\sigma(t,y)=\exp(y)$ would give a rough stochastic volatility model extending the class of models proposed by Gatheral, Jaisson and Rosenbaum \cite{GathJaiRosen2018}.  In their paper, an Ornstein-Uhlenbeck process driven by a fractional Brownian motion is shown to provide an excellent fit the volatility of stock prices. We extend this class of models to allow for a fractional time-derivative in the dynamics as well, opening for further flexibility in the modelling.   
Furthermore, we can define a simple rough Heston model as the variance process
$$
V(t):=Y^2(t)
$$
or, more generally, taking $n$ independent copies of $\mathbb R$-valued processes $W_i, i=1,\ldots,n$ driving $Y_i(t)$ as in \eqref{eq: representation mittag lefler},
$$
V(t)=\sum_{i=1}^n Y_i^2(t)
$$
Choosing $\sigma(t,v)=\sqrt{v}$ would give a rough Heston stochastic volatility, providing a possible extension of the class of models considered by 
El Euch and Rosenbaum \cite{EuchRosen}.

Let us return to a general separable Hilbert space $H$. 
Forward and futures prices can be realized as infinite dimensional stochastic processes, which call for operator-valued stochastic volatility models (see Benth, R\"udiger and S\"uss \cite{BRS} and Benth and Kr\"uhner \cite{BK-SIFINpaper}). To this end, let $H$ be the state space of the forward curves, given by some separable Hilbert space of real-valued functions on $\mathbb R_+$. 
We restrict to $\mathbb R_+$ as this plays the role of the time to maturity.
A possible (simplistic) model for the risk-neutral forward price at time $t\geq 0$ is defined as
\begin{equation}
df(t)=\partial_xf(t)dt+\Sigma(t)dB(t)
\end{equation}
where $B$ is some $H$-valued Wiener process with covariance operator $Q_B$. A direct extension of the rough stochastic volatility model could be the following: supposing that $H$ is a Banach algebra, we define
$$
\Sigma(t):=\exp(Y(t)).
$$
From the assumed algebra-structure of $H$, we can conclude that $\Sigma(t)$ is again an element of $H$. Moreover, $\Sigma(t)$ defines a linear operator on $H$, given as the multiplication operator $\Sigma(t)(f)=\Sigma(t)f, f\in H$. An example of a natural Hilbert space $H$ to use for modelling forward prices is the Filipovic space, which also happens to be a Banach algebra (see \cite{BK-COMSpaper}). The detailed knowledge of the covariance operator of $Y$
(recall Corollary \ref{cor:ML-covariance}) provides a starting point for empirical analysis of the volatility and its dependency across maturities for forward prices. For a fixed time to maturity, we will have a dynamics following a fractional stochastic volatility model similar to the one in \cite{GathJaiRosen2018} as discussed above. We refer to the recent paper \cite{AN} where clear evidence of rough stochastic volatility in commodity forward markets has been found (see also \cite{GathJaiRosen2018}). In particular, they show that for front month contracts, the roughness of the stochastic volatility is in general lower than for stock markets. Indeed, the authors find empirical evidence of Hurst parameters below 0.05 for metals and below 0.15 in other commodity markets. 

We can also introduce infinite dimensional extensions of the fractional Heston model. To this end, following Benth and Simonsen \cite{BS},
for some $H$-valued adapted process
$Z$ with $\vert Z(t)\vert_H=1$, define
$$
\Sigma(t):=Y(t)\otimes Z(t).
$$
Then, $\Sigma(t)(f)=\langle Y(t),f\rangle Z(t)$, and moreover, $\Sigma(t)^*=Z(t)\otimes Y(t)$. We take $\Sigma(t)$ as our infinite dimensional volatility process, where we observe that
$$
\Sigma(t)^*\Sigma(t)=Y^{\otimes 2}(t)
$$
I.e., $\Sigma(t)$ is in a sense the Cholesky decomposition of $Y^{\otimes 2}$. Notice that we use the convention $(f\otimes g)(h)=\langle f,h\rangle g$. 

We are concerned with the variance/volatility of elements like 
$$
U(t)=\mathcal L\int_0^t\Sigma(s)dB(s)
$$
where $\mathcal L\in H^*$, that is, a linear functional on $H$. If for any $x\geq 0$ the evaluation operator $e_x: f\mapsto e_xf:=f(x)$ is a continuous linear functional on
$H$\footnote{This is the case for the Filipovic space, say.}, we can think of $\mathcal L:=e_x$ as the noise process of the forward contract with time to maturity $x$. In power markets, say, the forwards deliver electricity over a settlement period. From e.g. \cite{BK-COMSpaper}, one finds that $\mathcal L$ in this case can be represented as some integral operator which is averaging over the maturities $x\geq 0$ in some domain (corresponding to the settlement period). 

The {\it total quadratic variation} of $U$ is given by the operator angle bracket process, see Cor. 8.17 in \cite{PesZab},
$$
\langle\langle U,U\rangle\rangle(t)=\int_0^t\mathcal L\Sigma(s)Q\Sigma(s)^*\mathcal L^*1ds.
$$
The {\it instantaneous quadratic variation} is the time-derivative of this expression, thus,
\begin{equation}
\sigma_{\mathcal L}^2(t):=\mathcal L\Sigma(t)Q\Sigma(t)^*\mathcal L^*1.
\end{equation}
The instantaneous quadratic variation is the stochastic variance process (that is, the squared volatility) of $U$, and has the form,
\begin{prop}
It holds
$$
\sigma_{\mathcal L}^2(t)=\vert\mathcal L(Y(t))\vert^2\vert Q_B^{1/2}Z(t)\vert_H^2
$$
\end{prop}
\begin{proof}
For $\mathcal T\in H^*$, we have
$$
\vert\mathcal T^*1\vert_H^2=\langle\mathcal T^*1,\mathcal T^*1\rangle=\mathcal T\mathcal T^*1.
$$
Hence,
$$
\sigma_{\mathcal L}^2(t)=\vert Q_B^{1/2}\Sigma(t)^*\mathcal L^*1\vert_H^2
$$
By definition,
$$
\Sigma^*(t)(f)=(Y(t)\otimes Z(t))(f)=\langle Y(t),f\rangle Z(t)
$$
Hence,
$$
Q_B^{1/2}\Sigma(t)^*(\mathcal L^*1)=\langle Y(t),\mathcal L^*1\rangle Q_B^{1/2}Z(t)=(\mathcal L Y(t))Q_B^{1/2}Z(t)
$$
The Proposition follows. 
\end{proof}
One may take $Z(t):=z$, with $\vert z\vert_H=1$. Thus, the stochastic variance is given as $\sigma_{\mathcal L}^2(t)=c\vert\mathcal L Y(t)\vert^2$, where $c$ is a scaling factor given by $c=\vert Q_B^{1/2}z\vert_H^2$. 

Let us now look at the stochastic process
$\vert\mathcal LY(t)\vert^2$. From Theorem \ref{thm:FDE} we find that $t\mapsto\mathcal LY(t)$ has paths which are $\rho=\gamma+\alpha-1$-regular, where  we recall that $\gamma\in(0,1)$ is the path regularity of $W$ and $\alpha\in(0,1)$ is the fractional derivative in the Ornstein-Uhlenbeck dynamics of $Y$. Moreover, $\gamma+\alpha>1$. Denoting by $v(t)$ the expected value of $\vert\mathcal L Y(t)\vert
^2$, we find from Corollary \ref{cor:ML-covariance}
$$
v(t)=\mathbb E[\vert\mathcal L Y(t)\vert^2]=\mathbb E[\langle Y(t),\mathcal L^*1\rangle_H^2]=\langle Q_Y(t,t)\mathcal L^*1,\mathcal L^*1\rangle_H=\mathcal L Q_Y(t,t)\mathcal L^*1.
$$
Thus, the expected moments of $\sigma_{\mathcal L}^2$ is given by
$$
\mathbb E[\sigma_{\mathcal L}^{2k}(t)]=c^k\mathbb E
[\vert\mathcal LY(t)\vert^{2k}]=c^k\xi_{2k} v(t)^k,
$$
where $\xi_{2k}$ is the $2k$th moment of a standard normal random variable, $k\in\mathbb N$. From Corollary \ref{cor:ML-covariance}, we have that $Q_Y\in\mathcal Q_{\eta}$ for $\eta<\beta+\alpha-1$ and $\beta>0$ such that
$\beta+\alpha>1$. We recover a fractional behaviour in the moments of the stochastic variance process similar to what has been observed empirically for a number of assets (see \cite{GathJaiRosen2018}) and more recently for commodity forwards \cite{AN} as noted above. 
In our context, we have the "roughness" split into a rough noise $W$ and a fractional derivative $\alpha$ which opens for a more
flexible modeling of the stochastic volatility. Moreover, we have provided an infinite-dimensional extension of the classical 
models.

\appendix

\section{Fractional Calculus}\label{sec:fractional calc}

We will apply certain elements from the theory of fractional calculus, involving  fractional derivatives and integrals of functions taking values in a separable Hilbert space $H$. Although there are several
concepts of fractional differentiation and integration, we will in this
article focus on fractional calculus of Riemann-Liouville
type.
\begin{defn}\label{fractional integral and derivative}
Let $f:\mathbb R_+\rightarrow H$ be a locally  Bochner  integrable function in $H$, that is  $f\in L^{1}_{\mathrm{loc}}(\mathbb R_+,H)$. We define the fractional integral of order $\alpha>0$  by 
\[
I^{\alpha}(f)(t)=\frac{1}{\Gamma\left(\alpha\right)}\int_{0}^{t}\left(t-s\right)^{\alpha-1}f(s)ds.
\] 
For $\alpha=0$, we set $I^0=I$ to be the identity operator. The fractional integral $I^\alpha$ is a linear operator on $L^1_{{\rm loc}}(\mathbb R_+,H)$. 
Furthermore, define the fractional derivative of order $\alpha$ by  
\begin{equation}\label{fractional derivative}
D^{\alpha}(f)(t)=\frac{d}{dt}I^{1-\alpha}(f)(t),
\end{equation}
 where the derivative $d/dt$ is interpreted in the Frechet sense whenever this exists.
\end{defn}
The next proposition is well known for the finite dimensional fractional Lebesgue integral, see for example \cite{Samko} for an comprehensive introduction to fractional calculus. 
\begin{prop}
\label{prop:properties of fractional calculus}Let $f\in L^{1}_{\mathrm{loc}}(\mathbb R_+,H)$. For any $t\in\mathbb{R}_+$, $\alpha\geq0$ and $\beta\geq0$, the fractional derivative and integral satisfy the following properties:
\begin{itemize}
    \setlength\itemsep{.1in}
    \item[{\rm (i)}]   $I^{\beta}(I^{\alpha}(f))(t)=I^{\alpha+\beta}(f)(t).$
    
    \item[{\rm (ii)}] $D^{\beta}(I^{\alpha}(f))(t)=I^{\alpha-\beta}(f)(t)$, whenever $f\in \mathcal C(\mathbb R_+,H)$ and $\beta\leq \alpha$. 
    
    \item[{\rm (iii)}]   $\sum_{i=0}^{\infty}I^{i\alpha}(1)=\sum_{i=0}^{\infty}\frac{t^{i\alpha}}{\Gamma(\alpha+1)}=:E_{\alpha,1} (t^{\alpha}).$
      
\end{itemize}
The function $E_{\alpha,\beta}(x):=\sum_{i\geq0}\frac{x^i}{\Gamma(i\alpha+\beta)}$ is known as the Mittag-Leffler function of
order $(\alpha,\beta)$.
\end{prop}

\begin{proof}
All of these identities are well known in the finite dimensional case, and we only give a sketch of the proof of their extension to the case of Hilbert valued elements here. The reader is referred to \cite{Samko} for a comprehensive discussion in the finite dimensional case. 
 First note that 
{\rm (i)} follows from an application of Fubini's theorem for Bochner integrals, together with elementary manipulations of the fractional integral.  To obtain {\rm (ii)} we apply {\rm (i)} together with the definition of the fractional derivative in \eqref{fractional derivative} to write 
\begin{equation*}
D^{\beta}(I^{\alpha}(f))(t)=\frac{d}{dt}I^{1+\alpha-\beta}(f)(t),
\end{equation*}
where the derivative is interpreted in the Frechet sense. 

First, consider the case $\alpha>\beta$ with $f\in L^1_{\text{loc}}(\mathbb{R}_+,H)$. From the triangle and the Bochner inequalities, we obtain the following estimate for arbitrary $h>0$ 
\begin{align*}
    &\frac{1}{h}\vert I^{1+\alpha-\beta}(f)(t+h)-I^{1+\alpha-\beta}(f)(t)-hI^{\alpha-\beta}(f)(t)\vert_H 
    \\
    &\leq  \frac{1}{\Gamma(1+\alpha-\beta)}\frac{1}{h} \int_{t}^{t+h}(t+h-s)^{\alpha-\beta}\vert f(s)\vert_Hds 
    \\
    &\quad+\frac{1}{\Gamma(1+\alpha-\beta)}\int_0 ^t \frac{1}{h}\left| (t+h-s)^{\alpha-\beta}-(t-s)^{\alpha-\beta}-h(\alpha-\beta)(t-s)^{\alpha-\beta-1}\right| \vert f(s)\vert_H ds.
\end{align*}
Since $t+h-s\leq h$ for $s\in[t,t+h]$, we find for the first integral on the right hand side that
$$
\frac1h\int_t^{t+h}(t+h-s)^{\alpha-\beta}\vert f(s)\vert_Hds\leq h^{\alpha-\beta}\frac1h\int_t^{t+h}\vert f(s)\vert_Hds.
$$
Thus, noting that the derivative of $t\mapsto(t-s)^{\alpha-\beta}$ is
$(\alpha-\beta)(t-s)^{\alpha-\beta-1}$, it holds from the fundamental theorem of calculus and the dominated convergence theorem that, $a.e.$,
\begin{equation*}
    \lim_{h\rightarrow 0} \frac{1}{h}\vert I^{1+\alpha-\beta}(f)(t+h)-I^{1+\alpha-\beta}(f)(t)-h I^{\alpha-\beta}(f)(t)\vert_H =0. 
\end{equation*}
Hence, {\rm (ii)} holds $a.e.$ for $f\in L^1_{\text{loc}}(\mathbb R_+,H)$ when $\alpha>\beta$. If $f$ is continuous, then the fundamental theorem of calculus holds everywhere, such that $h^{-1}\int_t^{t+h}\vert f(s)\vert_Hds
\rightarrow \vert f(t)\vert_H$ for all $t\in\mathbb R_+$ when $h\rightarrow 0$. 

In the case of $\alpha=\beta$, the result follows from the fundamental theorem of calculus for the Bochner integral.
This proves {\rm (ii)}. 
 At last, from a simple computation we find $I^{i\alpha}(1)(t)=\Gamma(i\alpha+1)^{-1}t^{i\alpha}$. Thus, by summing over $i$, property {\rm (iii)} holds. 
\end{proof}
\begin{rem}
Note that from point {\rm (ii)} of Proposition \ref{prop:properties of fractional calculus}, we see that $D^{\alpha}$ is the inverse operator
with respect to $I^{\alpha}$. We further remark that we must assume continuity of $f$ to make use of the fundamental theorem of calculus for Bochner integrals. 
\end{rem}

\section{Auxiliary proofs}\label{A: auxiliary proofs}
In this Appendix we have collected proofs from some of the claims in Section \ref{sect:applications}. 

\begin{proof}[Proof of Proposition \ref{Fubinis prop}]
Since $s\mapsto G(t,s)$ is integrable on $[0,t]$ and $s\mapsto Z(s)$ is continuous, the left-hand side of \eqref{fubini} makes sense as a Riemann integral. Furthermore, setting
$$
\Xi^\tau(t,s)= \int_s^\tau G(\tau,r)K(r,s)dr (W(t)-W(s)),
$$ 
where the integral term is interpreted in Bochner sense,
we give meaning to the integral on the right-hand side of \eqref{fubini} as a Volterra Young integral by application of Lemma \ref{lem:(Volterra-sewing-lemma)}. To this end, 
we check that $L(\tau,s):=\int_s^\tau G(\tau,r)K(r,s)dr$ is contained $\cK_\zeta$ for some $\zeta\in (0,1)$ such that $\gamma-\zeta>0$, and then apply Proposition  \ref{regularity W}. 
Using that $G\in \cK_\kappa$ for $\kappa\in (0,1)$, and $K\in \cK_\eta$, it is readily checked that 
\begin{align*}
    \|L(\tau,s)\|_{op} &\leq \|K\|_{\cK_\eta}\|G\|_{\cK_\kappa} \int_s^\tau (\tau-r)^{-\kappa}(r-s)^{-\eta}dr
    \\
    &\lesssim (\tau-s)^{1-\kappa-\eta}\int_0^1(1-\theta)^{-\kappa}\theta^{-\eta}d\theta. 
\end{align*}
Due to the assumption that both $\kappa,\eta\in (0,1)$ it follows that $|\int_0^1(1-\theta)^{-\kappa}\theta^{-\eta}d\theta|<\infty$. Since $\kappa\in (0,1)$, we have that $(\tau-s)^{1-\kappa-\eta}\lesssim (\tau-s)^{-\eta}$, and it therefore follows that $\|L\|_{\eta,1}<\infty$.   
By similar techniques, one can show $\|L\|_{\eta,i}<\infty$ for $i=2,3,4$, and thus we conclude that $L\in \cK_{\eta}$ inherits the regularity $K$.  Since $\gamma-\eta>0$ we conclude by Proposition \ref{regularity W} that the integral on the right-hand side of \eqref{fubini} is well defined in a Young-Volterra sense. 

We will now show that these two integrals are in fact equal. Let now $\cP$ denote a partition of $[0,t]$ and for some $s\in [0,t]$ let $\cP\cap[0,s]$ be the partition restricted to the domain $[0,s]$. 
Then it is readily seen that following relation holds
\begin{multline}\label{fubini sum}
    \sum_{[u,v]\in \cP}  \sum_{[u',v']\in \cP\cap[0,u]} G(t,u)[K(u,u')(W(v')-W(u'))](v-u) 
    \\
    =\sum_{[u',v']\in \cP} \sum_{[u,v]\in \cP\cap[v',t]} [G(t,u)K(u,u')(v-u)](W(v')-W(u'))
\end{multline}
Since the left-hand side of \eqref{fubini} is a Riemann integral, it is constructed as the limit of the sum on the left-hand side above. Similarly, the integral on the right-hand side of \eqref{fubini} is the limit of the sum appearing on the right-hand side above. Since \eqref{fubini sum} holds for arbitrary partitions $\cP$, it follows that \eqref{fubini} holds.

 \end{proof}

 \begin{proof}[Proof of Corollary \ref{cor:Fubini with cont path}]
An application of Proposition \ref{Fubinis prop} where we set $G(t,s)=(t-s)^{\alpha-1}$ and $K(t,s)=(t-s)^{-\eta}$ (both interpreted as multiplication of scalar operators on $H$), reveals that  for $\alpha\in(0,1)$  we have 
\begin{equation}\label{number three}
I^{\alpha}(X)(t)
=\Gamma\left(\alpha\right)^{-1}\Gamma\left(1-\eta\right)^{-1}\int_{0}^{t}\int_{r}^{t}\left(t-s\right)^{\alpha-1}\left(s-r\right)^{-\eta}dsdW(r).
\end{equation}
Next we will simplify the integral with respect to $ds$  in  \eqref{number three}, and to this end, we do a change of variables by setting $s=r+\theta\left(t-r\right)$ and compute 
\begin{equation*}
\int_{r}^{t}\left(t-s\right)^{\alpha-1}\left(s-r\right)^{-\eta}ds=\left(t-r\right)^{\alpha-\eta}B(\alpha,1-\eta),
\end{equation*}
where $B$ is the Beta-function. Using the properties of the Beta-function, we further compute that 
\begin{equation}
\frac{\Gamma\left(\alpha-\eta+1\right)}{\Gamma\left(\alpha\right)\Gamma\left(1-\eta\right)}\int_{0}^{t}\int_{r}^{t}\left(t-s\right)^{\alpha-1}\left(s-r\right)^{-\eta}dsdW(r)
=\int_{0}^{t}\left(t-r\right)^{\alpha-\eta}dW(r).
\end{equation}
Therefore, it holds 
\begin{align*}
\sum_{i=0}^{\infty}A^{\circ i}I^{i\alpha}\left(X\right)(t)&=\int_{0}^{t}\left(t-r\right)^{-\eta}\sum_{i=0}^{\infty}\frac{A^{\circ i}\left(t-r\right)^{i\alpha}}{\Gamma\left(i\alpha-\eta+1\right)}dW(r) \\
&=\int_{0}^{t}\left(t-r\right)^{-\eta}E_{\alpha,1-\eta}\left(A\left(t-r\right)^{\alpha}\right)dW(r),
\end{align*}
where we have applied the linearity of the  integral together with the pathwise Fubini Theorem in Proposition \ref{Fubinis prop} to exchange the infinite sums with integration.   
\end{proof}

 \begin{proof}[Proof of Theorem \ref{thm:FDE}.]
The proof is given in two steps. First we show the existence of a solution by showing that \eqref{eq: representation mittag lefler} satisfies \eqref{eq: fractional equation}, and next we show that this solution is in fact unique. 
An application
of Corollary \ref{cor:Fubini with cont path} shows that (\ref{eq: representation mittag lefler})  satisfies \eqref{eq: fractional equation}.  Indeed, define 
\begin{equation}\label{eq: sol rep}
Y(t):=E_{\alpha,1}\left(At^{\alpha}\right)y+\int_{0}^{t}\left(t-s\right)^{\alpha-1}E_{\alpha,\alpha}\left(A\left(t-s\right)^{\alpha}\right)dW(s).
\end{equation}
Inserting (\ref{eq: sol rep}) into $AI^\alpha(Y)(t)$ yields
\begin{align}\nonumber
    AI^\alpha(Y)(t)&=AI^\alpha\left(E_{\alpha,1}\left(A(\cdot)^{\alpha}\right)\right)y+AI^\alpha\left(\int_{0}^{\cdot}\left(\cdot-s\right)^{\alpha-1}E_{\alpha,\alpha}\left(A\left(\cdot-s\right)^{\alpha}\right)dW(s)\right)(t)
    \\\label{eq:I1 and I2}
    &=:I_1(t)+I_2(t).
\end{align}
By definition of $E_{\alpha,1}$  (see Corollary \ref{cor:Fubini with cont path}), we observe that  
\begin{equation}\label{eq: I1}
I_1(t)=E_{\alpha,1}\left(At^\alpha\right)y-y.
\end{equation}
Moreover, using the convolution property of the fractional integral $I^\alpha$  together with the representation of the Mittag-Leffler function, followed by the Fubini property from Corollary \ref{cor:Fubini with cont path}, we find \begin{equation}\label{eq: I2}
    I_2 (t)=\int_{0}^{t}\left(t-s\right)^{\alpha-1}E_{\alpha,\alpha}\left(A\left(t-s\right)^{\alpha}\right)dW(s)-X(t)
\end{equation}
Thus, combining \eqref{eq: I1} and \eqref{eq: I2} and inserting into \eqref{eq:I1 and I2}, it is clear that 
\begin{equation*}
    AI^\alpha(Y)(t)=E_{\alpha,1}\left(At^\alpha\right)y-y+\int_{0}^{t}\left(t-s\right)^{\alpha-1}E_{\alpha,\alpha}\left(A\left(t-s\right)^{\alpha}\right)dW(s)-X(t).
\end{equation*}
After rearranging this expression it follows that \eqref{eq: representation mittag lefler} solves \eqref{eq: fractional equation}. 

We continue to prove that this solution is unique. Suppose $Y$ and $\tilde{Y}$ are two solutions in $C^{\rho}([0,T],H)$ to \eqref{eq: fractional equation}, starting in $y\in H$ and $\tilde{y}\in H$ respectively. Then it is readily checked that 
\begin{equation}\label{ineq: Y-tY}
    |Y(t)-\tilde{Y}(t)|_H\leq |y-\tilde{y}|_H +AI^\alpha (|Y-\tilde{Y}|_H)(t). 
\end{equation}
By application of the fractional Gr\"onwall lemma \cite[Corollary 2]{YeGaoDing2007}, Inequality \eqref{ineq: Y-tY} implies that 
\begin{equation*}
    |Y(t)-\tilde{Y}(t)|_H \leq |y-\tilde{y}| E_{\alpha,1}(A \Gamma(\alpha)t^\alpha). 
\end{equation*}
Thus if $y-\tilde{y}=0$, we see that $|Y(t)-\tilde{Y}(t)|_H=0$, which implies that the solution is unique. This concludes the  proof. 

\end{proof}

\bibliographystyle{siam}

\end{document}